\documentclass[12pt]{amsart}

\makeatletter
\def\section{\@startsection{section}{1}%
  \z@{1.2\linespacing\@plus\linespacing}{.5\linespacing}%
  {\normalfont\scshape\centering}}
\def\subsection{\@startsection{subsection}{2}%
  \z@{0.9\linespacing\@plus.7\linespacing}{-.5em}%
  {\normalfont\bfseries}}
\makeatother

\usepackage{amsmath}
\usepackage{amsfonts}
\usepackage{amsbsy}
\usepackage{amssymb}
\usepackage{times}
\usepackage{mathrsfs}
\usepackage{exscale}
\usepackage{graphicx}
\usepackage[colorlinks,citecolor=marin,linkcolor=rouge,
            bookmarksopen,
            bookmarksnumbered
           ]{hyperref}

\usepackage{times}

\usepackage{color}
\definecolor{gr}{rgb}   {0.,   0.69,   0.23 } 
\definecolor{mg}{rgb}   {0.85,  0.,    0.85}

\newcommand{\Bk}{\color{black}}

\definecolor{marin}{rgb}   {0.,   0.65,   0.25}
\definecolor{rouge}{rgb}   {0.8,   0.,   0.}

\newtheorem{theorem}{Theorem}[section]
\newtheorem{lemma}[theorem]{Lemma}

\theoremstyle{definition}

\theoremstyle{remark}
\newtheorem{remark}[theorem]{Remark}

\numberwithin{equation}{section}

\newcommand{\ee}{\hskip0.15ex}

\newcommand{\dd}[1]{_{\raise-0.3ex\hbox{$\scriptstyle #1$}}}

\newcommand {\Norm}[2]{ \mathchoice 
    {|\ee #1\ee|\dd{#2}}
    {| #1 |_{#2}}
    {| #1 |_{#2}}
    {| #1 |_{#2}} }
\newcommand {\DNorm}[2]{ \mathchoice 
    {\|\ee #1\ee\|\dd{#2}}
    {\| #1 \|_{#2}}
    {\| #1 \|_{#2}}
    {\| #1 \|_{#2}} }
\newcommand {\Normc}[2]{ \mathchoice 
    {|\ee #1\ee|\dd{#2}^2}
    {| #1 |_{#2}^2}
    {| #1 |_{#2}^2}
    {| #1 |_{#2}^2} }
\newcommand {\DNormc}[2]{ \mathchoice 
    {\|\ee #1\ee\|\dd{#2}^2}
    {\| #1 \|_{#2}^2}
    {\| #1 \|_{#2}^2}
    {\| #1 \|_{#2}^2} }

\renewcommand{\div}{\operatorname{\rm div}}
\newcommand{\grad}{\operatorname{\textbf{grad}}}
\newcommand{\curl}{\operatorname{\textbf{curl}}}

\newcommand\Sbb{{\mathbb S}}
\newcommand\C{{\mathbb C}}
\newcommand\R{{\mathbb R}}
\newcommand\N{{\mathbb N}}
\newcommand\Q{{\mathbb Q}}

\newcommand\Z{{\mathbb Z}}

\newcommand\sC{{\mathscr C}}

\renewcommand\L{{\mathscr L}}

\newcommand\OO{{\mathcal O}}

\renewcommand\SS{{\mathcal S}}

\newcommand\WW{{\boldsymbol W}}

\newcommand{\bc}{{\boldsymbol{c}}}
\newcommand{\be}{{\boldsymbol{e}}}
\newcommand{\bff}{\boldsymbol{f}}
\newcommand{\bu}{\boldsymbol{u}}
\newcommand{\bv}{\boldsymbol{v}}
\newcommand{\bw}{\boldsymbol{w}}
\newcommand{\bx}{{\boldsymbol{x}}}

\newcommand {\gA}{\mathfrak{A}}
\newcommand {\gC}{\mathfrak{C}}
\newcommand {\gE}{\mathfrak{E}}

\newcommand {\gM}{{\mathfrak M}}
\newcommand {\gP}{{\mathfrak P}}
\newcommand {\gR}{{\mathfrak R}}

\newcommand {\gU}{{\mathfrak U}}
\newcommand {\gW}{{\mathfrak W}}

\newcommand{\inff}{\mathop{\operatorname{\vphantom{p}inf}}}

\renewcommand{\Re}{\operatorname{\rm Re}}
\renewcommand{\Im}{\operatorname{\rm Im}}

\begin{document}

\title{\bf The inf-sup constant for the divergence on corner domains  }  

\author{Martin Costabel, Michel Crouzeix, Monique Dauge and Yvon Lafranche}

\address{IRMAR UMR 6625 du CNRS, Universit\'{e} de Rennes 1 }

\address{Campus de Beaulieu,
35042 Rennes Cedex, France \bigskip}

\email{martin.costabel@univ-rennes1.fr}
\urladdr{http://perso.univ-rennes1.fr/martin.costabel/}

\email{michel.crouzeix@univ-rennes1.fr}
\urladdr{http://perso.univ-rennes1.fr/michel.crouzeix/}

\email{monique.dauge@univ-rennes1.fr}
\urladdr{http://perso.univ-rennes1.fr/monique.dauge/}

\email{yvon.lafranche@univ-rennes1.fr}
\urladdr{http://perso.univ-rennes1.fr/yvon.lafranche/}

\date{10 January 2014}  



\begin{abstract}
The inf-sup constant for the divergence, or LBB constant, is related to the Cosserat spectrum. It has been known for a long time that on non-smooth domains the Cosserat operator has a non-trivial essential spectrum, which can be used to bound the LBB constant from above. We prove that the essential spectrum on a plane polygon consists of an interval related to the corner angles and that on three-dimensional domains with edges, the essential spectrum contains such an interval. We obtain some numerical evidence for the extent of the essential spectrum on domains with axisymmetric conical points by computing the roots of explicitly given holomorphic functions related to the corner Mellin symbol. Using finite element discretizations of the Stokes problem, we present numerical results pertaining to the question of the existence of eigenvalues below the essential spectrum on rectangles and cuboids.
\end{abstract}

\maketitle

\section{Introduction}

The inf-sup constant of the divergence $\beta(\Omega)$ is defined for a domain $\Omega\subset\R^{d}$ as
\begin{equation}
\label{E:infsup}
    \beta(\Omega) = \ {\inff_{q\ee\in\ee L^2_\circ(\Omega)}} \ \ 
   {\sup_{\bv\ee\in\ee
   H^1_0(\Omega)^d}} \ \ 
   \frac{\big\langle\! \div\bv, q \big\rangle_\Omega}{\Norm{\bv}{1,\Omega}\,\DNorm{q}{0,\Omega}} \ .
\end{equation}
Here 
$L^2_\circ(\Omega)$ denotes the space of square integrable functions with mean value zero,  with norm $\DNorm{\cdot}{0,\Omega}$ and scalar product $\big\langle \cdot, \cdot \big\rangle_\Omega$, and 
$H^1_0(\Omega)$ is the usual Sobolev space, closure of the space of smooth functions of compact support in $\Omega$ with respect to the $H^{1}$ seminorm, which for vector-valued functions is defined by
\[
   \Norm{\bv}{1,\Omega} = \DNorm{\grad \bv}{0,\Omega} =
   \Big(\sum_{k=1}^d \sum_{j=1}^d \DNormc{\partial_{x_j}v_k}{0,\Omega} \Big)^{1/2} .
\]
The inf-sup condition $\beta(\Omega)>0$ is also called LBB condition \cite{Malkus1981} or Ladyzhenskaya-Babu\v{s}ka-Brezzi condition (a term coined ca.\ 1980 by T. J. Oden, on suggestion by J. L. Lions \cite{Oden_perso}), and $\beta(\Omega)$ is therefore also known as LBB constant \cite{Dobrowolski2003,Dobrowolski2005}. It plays an important role in the discussion of the stability of solutions and numerical approximations of the equations of hydrodynamics. The exact value of the constant is known only for a small class of domains, first for balls and ellipsoids (derived from the Cosserat spectrum \cite{Cosserat1898,Cosserat1898b}), then in three dimensions for spherical shells \cite{Cosserat1902,LiuMarkenscoff1998}, and finally in two dimensions for annular domains \cite{ChizhOlsh00} and some domains defined by simple conformal images of a disk \cite[\S5]{Horgan1975}, \cite{HorganPayne1983,Zsuppan2004}. Its precise value remains, however, unknown for such simple domains as a square. 

Finding estimates, from below and from above, for $\beta(\Omega)$ has therefore been an important subject for many years. The question can be rephrased in terms of the Cosserat eigenvalue problem, see relation \eqref{E:LBB-Coss} below, and it is this problem that we will study in the present paper. In particular, any number known to belong to the nontrivial part of the spectrum of the Cosserat operator will imply an upper bound for the LBB constant.

Other inequalities and eigenvalue problems are known to be related to the Cosserat eigenvalue problem and the LBB condition, namely the Korn and Friedrichs inequalities. In two dimensions in particular, such relations have been investigated since the paper \cite{Friedrichs1937} by Friedrichs, and the equivalence between these inequalities and equations between the corresponding constants have been studied in the classical paper \cite{HorganPayne1983} by Horgan and Payne, see \cite{CoDa_IBAFHP} for improved versions of some of their results.

For the Cosserat eigenvalue problem, the values $0$ and $1$ are eigenvalues of infinite multiplicity, and $1/2$ is an accumulation point of eigenvalues. On smooth domains, these are the only points in the essential spectrum, as shown by Mikhlin \cite{Mikhlin1973}, see \cite{Crouzeix1997} for a complete proof for $\sC^{3}$ domains. On non-smooth domains, a non-trivial essential spectrum is present, as was already pointed out in \cite{Friedrichs1937} for the equivalent Friedrichs problem.

We show in Theorem~\ref{T:esspolyg} for two-dimensional domains that a polygonal corner of opening $\omega\in(0,2\pi)$ contributes an interval
$[\frac12-\frac{\sin\omega}{2\omega}, \frac12+\frac{\sin\omega}{2\omega}]$ to the essential spectrum of the Cosserat operator. This has been known for a while already \cite{CoDaLausanne2000}, but the proof has not yet been published. A corollary is an upper bound for the LBB constant, see also \cite{Stoyan1999,Stoyan2001}
\begin{equation}
\label{E:betamaj}
  \beta(\Omega)\le\sqrt{\frac12-\frac{\sin\omega}{2\omega}}\,.
\end{equation}
For the square $\Omega=\Box$, in \cite[(6.39)]{HorganPayne1983} a conjecture was offered that amounts to
$\beta(\Box)=\sqrt{2/7}=0.53\dots$, which is obviously incompatible with the upper bound 
$\beta\le\sqrt{\frac12-\frac1\pi}=0.42\dots$ from \eqref{E:betamaj}.
Although already conjectured in \cite{Stoyan1999}, it is still unknown whether this latter inequality is an equality or whether it is strict, that is, whether for the square there exist Cosserat eigenvalues below the minimum of the essential spectrum. In Section~\ref{S:Rectangles} below, we present numerical evidence suggesting that there are no such eigenvalues. But this is not yet proven.

For three-dimensional domains with edges, the two-dimensional corner domain transversal to the edge implies an inclusion of the corresponding interval in the essential spectrum, which therefore always contains such an interval symmetric with respect to the point $1/2$ (see Section~\ref{Ss:Edges}). This is very different from the case of smooth domains, where there exists the known example of the ball that has a Cosserat spectrum consisting (apart from the trivial point $1$) of a sequence of eigenvalues 
$\sigma_{k}=\frac{k}{2k+1}$, $k\ge1$, 
converging to $1/2$ from below and therefore has no spectrum in the interval $(1/2,1)$\cite{Cosserat1898b,Crouzeix1997,SimadervWahl2006}.

For three-dimensional bounded domains having conical boundary points with tangent cones of revolution, we present in Section~\ref{Ss:Cones} numerical results showing that there are also intervals contained in the essential spectrum of the Cosserat operator.

The Cosserat eigenvalue problem as originally formulated by E. and F. Cosserat \cite{Cosserat1898} can be written as follows: 
Find nontrivial $\bu\in H^{1}_{0}(\Omega)^{d}$ and $\sigma\in\C$ such that
\begin{equation}
\label{E:Coss1}
 \sigma\Delta \bu - \nabla\div \bu=0 \,.
\end{equation}
This is the spectral problem of the bounded selfadjoint operator 
$\Delta^{-1}\nabla\div$ on $H^{1}_{0}(\Omega)^{d}$, where $\Delta^{-1}$ is the inverse of the Laplace operator with Dirichlet conditions 
$\Delta: H^{1}_{0}(\Omega)\to H^{-1}(\Omega)$.

On the orthogonal complement of the kernel of $\div$, which is the eigenspace for the trivial eigenvalue $\sigma=0$, this operator is equivalent to the operator
\begin{equation}
\label{E:SS}
  \SS = \div\Delta^{-1}\nabla\,:\; L^{2}_{\circ}(\Omega) \to L^{2}_{\circ}(\Omega)\,,
\end{equation}
the Schur complement operator of the Stokes system. 

We define the \emph{Cosserat constant} $\sigma(\Omega)$ of the domain $\Omega$ as the minimum of the spectrum of the operator $\SS$.

It is then an exercise in elementary Hilbert space theory to show that there holds
\begin{equation}
\label{E:LBB-Coss}
 \sigma(\Omega) = \beta(\Omega)^{2}\,.
\end{equation}
To see this, write the LBB constant using the definition of the $H^{-1}$ norm and the fact that $\Delta$ is an isometry from $H^{1}_{0}(\Omega)$ to $H^{-1}(\Omega)$:
$$
 \beta(\Omega)^{2}
  = \inff_{q\ee\in\ee L^2_\circ(\Omega)}
   \frac{\Normc{\nabla q}{-1,\Omega}}{\DNormc{q}{0,\Omega}}
  = \inff_{q\ee\in\ee L^2_\circ(\Omega)}
   \frac{\big\langle {-}\Delta^{-1}\nabla q, \nabla q \big\rangle_\Omega}{\DNormc{q}{0,\Omega}}
  = \inff_{q\ee\in\ee L^2_\circ(\Omega)}
   \frac{\big\langle \SS q, q \big\rangle_\Omega}{\DNormc{q}{0,\Omega}}
  = \sigma(\Omega)
\,.
$$

For numerical approximations of the Cosserat eigenvalue problem, we will use the equivalent formulation as an eigenvalue problem for the Stokes system:

Find $\bu\in H^{1}_{0}(\Omega)^{d}$, 
$p\in L^{2}_{\circ}(\Omega) \setminus\{0\}$, $\sigma\in\C$ such that
\begin{gather}
\label{E:Stokesp}
\begin{aligned}
-\Delta \bu + \nabla p &= 0 \\
 \div\bu &=\sigma p \;.
\end{aligned}
\end{gather}

In the following, we will mainly discuss the essential spectrum of the Cosserat problem. 
This is the set of $\sigma\in\C$ such that the operator
\begin{equation}
\label{E:Ls}
  L_{\sigma} = \sigma\Delta - \nabla\div \,:\;
  H^{1}_{0}(\Omega)^{d} \to H^{-1}(\Omega)^{d}
\end{equation}
is not a Fredholm operator. 

From the variational formulation of the operator $L_{\sigma}$ for $\bu,\bv\in H^{1}_{0}(\Omega)^{2}$
\begin{equation}
\label{E:Lsweak}
  \langle L_{\sigma}\bu, \bv\rangle =
  -\sigma\int_{\Omega}\nabla\bu :\! \nabla\bv + \int_{\Omega}\div\bu\,\div\bv
\end{equation}
it is not hard to see that such $\sigma$ have to be real, between $0$ and $1$, and that $L_{\sigma}$ is Fredholm if and only if it is Fredholm of index $0$ and if and only if it is semi-Fredholm.

\section{Domains with conical points}\label{S:Conic}
We will show how Kondrat'ev's classical theory \cite{Kondratev67} applies to the Cosserat operator $L_\sigma$ when $\Omega$ is a \emph{domain with conical points} in $\R^d$. This means that the boundary of $\Omega$ is smooth except in a finite set $\gC$ of points $\bc$, called the corners of $\Omega$, and in the neighborhood of each corner $\bc$ the domain $\Omega$ is locally diffeomorphic to a \emph{regular cone} $\Gamma_\bc\,$, i.e., the section $G_\bc=\Gamma_\bc\cap\Sbb^{d-1}$ is a smooth domain in $\Sbb^{d-1}$.

From the discussion in \cite[\S2]{Mikhlin1973}, we know that the system $\sigma\Delta - \nabla\div$ is elliptic at any point of $\Omega$ if (and only if) $\sigma\not\in\{0,1\}$, and that the Dirichlet boundary condition covers $\sigma\Delta - \nabla\div$ at any smooth point of $\partial\Omega$ if, moreover, $\sigma\neq\frac12$.

In contrast with the smooth case when $0,\frac12,1$ are the only values for which $L_\sigma$ is not Fredholm, the corners of $\Omega$ cause an enlargement of this exceptional set, in general.

\subsection{The Mellin symbol of the Cosserat operator}\label{Ss:Mellin}
For determining the values $\sigma$ such that $L_{\sigma}$ in \eqref{E:Ls} is not Fredholm, we use Kondrat'ev's \cite{Kondratev67} technique of corner localization and Mellin transformation. 

Let us choose a corner $\bc$ and write $(r,\vartheta)\in\R_+\times G_\bc$ for the polar coordinates in the tangent cone $\Gamma_\bc$. The homogeneous Lam\'e system $\sigma\Delta - \nabla\div$ can be written as
\[
   \sigma\Delta - \nabla\div = r^{-2} \L_\sigma(\vartheta;r\partial_r,\partial_\vartheta)
\] 
where the $d\times d$ system $\L_\sigma$ has coefficients independent of $r$. The Mellin transformation $u\mapsto \int_0^\infty r^{-\lambda-1}u(r,\vartheta)\,dr$ transforms $r\partial_r$ into the multiplication by $\lambda$. The \emph{Mellin symbol} $\gA_\sigma^\bc$ at the corner $\bc$ is the operator valued function (known as \emph{operator pencil} in the Russian literature) defined as
\begin{equation}
\label{E:MellinS}
   \gA_\sigma^\bc(\lambda) = \L_\sigma(\vartheta;\lambda,\partial_\vartheta)\,:\;
   H^{1}_{0}(G_\bc)^{d} \to H^{-1}(G_\bc)^{d} ,\quad\lambda\in\C.
\end{equation}

\subsection{The Fredholm theorem}\label{Ss:Fredholm}
In Kondrat'ev's theory, the Fredholm property of $L_\sigma$ is studied in the framework of  weighted Sobolev spaces \cite{Kondratev67}. In the book \cite{KozlovMazyaRossmann97b} by Kozlov-Maz'ya-Rossmann, these spaces are defined as follows: For $\ell\in\N$ and $\beta\in\R$
\begin{equation}
\label{E:V}
   V^\ell_\beta(\Omega) = \{u\in L^2_{\rm loc}(\Omega):\ \  
   |\bx-\bc|^{\beta+|\alpha|-\ell}\partial^\alpha_\bx u\in L^2(\Omega)\ 
   \forall\bc\in\gC,\ \forall\alpha\in\N^d,\ |\alpha|\le\ell\}.
\end{equation}

\begin{lemma}
\label{L:H10}
Let $\Omega\subset\R^d$ be a domain with conical points.
The space $H^1_0(\Omega)$ coincides with the subspace of $V^1_0(\Omega)$ consisting of functions with zero boundary traces.
\end{lemma}

This lemma is proved using a Poincar\'e inequality in angular variables, see \cite[Lemma 6.6.1]{KozlovMazyaRossmann97b}. Then \cite[Theorem 6.3.3]{KozlovMazyaRossmann97b} implies the following result.

\begin{theorem}
\label{T:Mellin}
Let $\Omega\subset\R^d$ be a domain with conical points.
Assume that $\sigma$ does not belong to $\{0,\frac12,1\}$ and that for each corner $\bc$ and all complex numbers $\lambda$ with real part $1-\frac d2$, the Mellin symbol $\gA_\sigma^\bc(\lambda)$ is invertible. Then $L_\sigma$ is Fredholm.
\end{theorem}

Note that we cannot apply \cite[Theorem  3.1]{Kondratev67} right away because it is assumed there that the Sobolev exponent $\ell$ is at least $2$. For this reason we have to use the extension performed in \cite{KozlovMazyaRossmann97b}. The assumption $\sigma\not\in\{0,\frac12,1\}$ ensures that the boundary value problem $L_\sigma$ is elliptic.

Concerning the critical abscissa $1-\frac d2$, there is a simple rule of thumb for determining it: A corner $\bc$ being chosen together with a non-zero angular function $W\in H^1_0(G_\bc)$ and a compactly supported cut-off function $\chi$ equal to $1$ near $\bc$, we find that $1-\frac d2$ is the smallest real number $\eta$ such that all functions of the form $\chi\, r^{\lambda}W(\vartheta)$ belong to $H^1_0(\Gamma_\bc)$, whenever $\Re\lambda>\eta$.

\section{Plane polygons}\label{S:Poly}

Let $\Omega\subset\R^{2}$ be a polygon, that is, a bounded Lipschitz domain the boundary $\partial\Omega$ of which consists of a finite set of straight segments. This set being chosen in a minimal way, the corners $\bc$ of $\Omega$ are the ends of these segments. Each corner $\bc$ belongs to two neighboring segments and the tangent cone $\Gamma_\bc$ is a plane sector, the opening of which is denoted by $\omega_\bc$. Thus the section $G_\bc$ can be identified with the interval $(-\frac{\omega_\bc}{2},\frac{\omega_\bc}{2})$.

\subsection{The Mellin determinant}\label{Ss:Mdet}

In two-dimension, the invertibility of the Mellin symbol $\gA_\sigma^\bc(\lambda)$, which is a $2\times2$ Sturm-Liouville system on an interval with polynomial dependence on $\lambda$, can be further reduced to the non-vanishing of a scalar holomorphic function, which we may call Mellin determinant $\gM_\sigma^\bc(\lambda)$. This Mellin determinant is well known for the case of the Dirichlet problem of the Lam\'e system of linear elasticity. It is constructed and described in detail in the book \cite[Chapter 3]{KozlovMazyaRossmann01}, where also references to earlier works can be found.
We can use the results of these calculations, simply by noticing that the operator $L_{\sigma}$ \emph{is} the Lam\'e operator, if we set
\[
  \sigma
  = 2\nu-1
\]
with 
the Poisson ratio $\nu$.
What is non-standard here, compared to the discussion in \cite{KozlovMazyaRossmann01}, is first that the $\sigma$ considered here correspond to the ``non-physical'' range 
$\frac12<\nu<1$,
and second that we consider the question of loss of $H^{1}$ regularity, that is, zeros of the Mellin determinant on the critical line $\{\Re\lambda=1-\frac d2=0\}$.

In \cite[(3.1.22/23)]{KozlovMazyaRossmann01}, the Mellin determinant for a corner of opening angle $\omega$ is given as 
$$
 \gM(\lambda)= 
 \lambda^{-2}\big( (3-4\nu)^{2}\sin^{2}\lambda\omega - \lambda^{2}\sin^{2}\omega \big)\,.
$$
The characteristic equation $\gM(\lambda)=0$ can therefore be written as the two equations
\begin{equation}
\label{E:char}
   (1-2\sigma)\frac{\sin\lambda\omega}{\lambda} = 
   \pm\sin\omega.
\end{equation}

In the following paragraph \ref{Ss:SingFun} we present a new and straightforward way to calculate $\gM$ together with the \emph{singular functions} associated with the roots $\lambda$ of $\gM$, i.e.\ the solutions $\bw_\lambda$ of the equation $L_\sigma\bw=0$ that are homogeneous of degree $\lambda$.

\begin{remark}
\label{R:sings}
These singular functions play a double role here: 
If $\Re\lambda=0$, they do not belong to $H^{1}_{0}$ near the corner, which is the reason why $L_{\sigma}$ is then not a Fredholm operator and $\sigma$ belongs to the essential spectrum, see Theorems~\ref{L:charpolyg} and~\ref{T:esspolyg}. Figure~\ref{F3.1} shows this case: As a function of $\sigma$ and $\omega$, the imaginary part of $\lambda$ is plotted. 
On the other hand, if $\Re\lambda>0$, then the singular functions $\bw_\lambda$ belong to $H^{1}_{0}$ near the corner and they describe the corner asymptotics of any solution in $H^{1}_{0}(\Omega)$ of $L_{\sigma}\bu=\bff$ with smooth $\bff$, in particular of Cosserat eigenfunctions for such eigenvalues $\sigma$ that are not in the essential spectrum:
\begin{equation}
\label{E:reg}
   \bu\in H^{1+s}(\Omega)^2,\quad\forall 
   s<\min\{\Re\lambda \mid \Re\lambda>0,\, \lambda\text{ root of \eqref{E:char} } \}
\end{equation}
If $\Re\lambda$ is close to zero, the derivatives of $\bu$ (thus the ``pressure'' part $p$ of the Stokes solution in \eqref{E:Stokesp}) will go to infinity quickly at the corner, which poses problems for the numerical approximation of the Cosserat eigenvalue problem, see the discussion in Section~\ref{S:Rectangles}. In Figure~\ref{F3.2}, we plot $\lambda$ as function of $\sigma$ and $\omega$ in this case. Note that for $0\le\sigma\le1$, $\lambda$ is either real or purely imaginary. 
\end{remark}

\begin{figure}
\centerline{\includegraphics[scale=0.41]{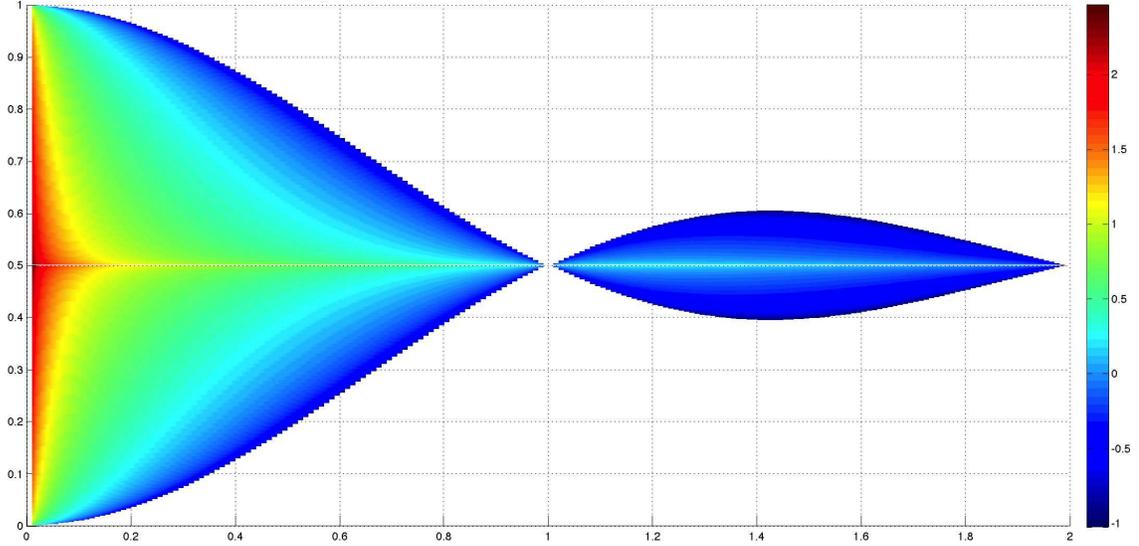}}
\caption{Color plot of the decimal logarithm of $|\Im\lambda|$ for purely imaginary roots $\lambda$ of the characteristic equation \eqref{E:char} as a function of the opening $\omega/\pi\in(0,2)$ (abscissa) and the Cosserat spectral parameter $\sigma\in[0,1]$ (ordinate).}
\label{F3.1}
\end{figure}

\begin{figure}
\centerline{\includegraphics[scale=0.41]{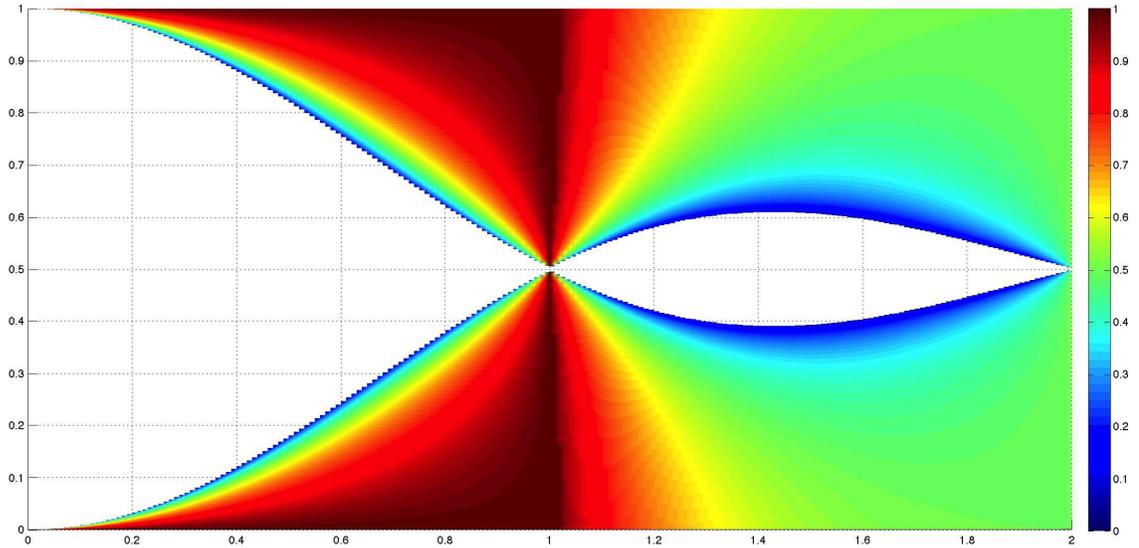}}
\caption{Color plot of smallest positive real roots $\lambda$ of the characteristic equation \eqref{E:char} as a function of the opening $\omega/\pi\in(0,2)$ (abscissa) and the Cosserat spectral parameter $\sigma\in[0,1]$ (ordinate).}
\label{F3.2}
\end{figure}

\subsection{Singular functions}\label{Ss:SingFun}
Let us choose a corner $\bc$ and drop the reference to that corner in the notation.
We assume for simplicity that $\bc$ is the origin.
The symbol $\gA_\sigma(\lambda)$ is not invertible iff it has a nonzero kernel. We note that for 
$\WW\in H^1_0(-\frac\omega2,\frac\omega2)^2$ we have the equivalence
\begin{equation}
\label{E:equiv}
   \WW\in\ker\gA_\sigma(\lambda)\quad \Longleftrightarrow\quad
   (\sigma\Delta - \nabla\div)(r^\lambda\WW(\theta))=0.
\end{equation} 
So we first solve
\begin{equation}
\label{E:sansdir}
   (\sigma\Delta - \nabla\div)(r^\lambda\WW(\theta)) = 0
\end{equation}
\emph{without boundary conditions} and in a second step find conditions on $\lambda$ so that there exist nonzero solutions satisfying the Dirichlet conditions at $\theta=\pm\frac\omega2$.

{\sc Step 1.} We know from \cite[Theorem 2.1]{CostabelDauge93c} that for any $\lambda\in\C$ the solutions of \eqref{E:sansdir} form a space of dimension 4, which we denote by $\gW_\sigma(\lambda)$. 
Moreover \cite[\S2.b(iii)]{CostabelDauge93c} tells us that if $\lambda\not\in\N$, it is sufficient to look for each component of the homogeneous function $r^\lambda\WW(\theta)$ in the space generated by $z^\lambda$, $\bar z^\lambda$, $z^{\lambda-1}\bar z$, and $\bar z^{\lambda-1}z$. Here we have identified $\R^{2}$ with the complex plane $\C$ by the formula $z=x_1+ix_2$. So the Ansatz space for $\bw=r^\lambda\WW(\theta)$ has the basis
\begin{gather*}
   \bw_{(1)} = \begin{pmatrix}1\\i\end{pmatrix} z^\lambda,\quad
   \bw_{(2)} = \begin{pmatrix}1\\i\end{pmatrix} \bar z^\lambda,\quad
   \bw_{(3)} = \begin{pmatrix}1\\i\end{pmatrix} z^{\lambda-1}\bar z,\quad
   \bw_{(4)} = \begin{pmatrix}1\\i\end{pmatrix} \bar z^{\lambda-1} z,\\
   \widetilde\bw_{(1)} = \begin{pmatrix}1\\-i\end{pmatrix} z^\lambda,\quad
   \widetilde\bw_{(2)} = \begin{pmatrix}1\\-i\end{pmatrix} \bar z^\lambda,\quad
   \widetilde\bw_{(3)} = \begin{pmatrix}1\\-i\end{pmatrix} z^{\lambda-1}\bar z,\quad
   \widetilde\bw_{(4)} = \begin{pmatrix}1\\-i\end{pmatrix} \bar z^{\lambda-1} z.
\end{gather*}
Applying $L_\sigma\bw=0$ to these vector functions is an easy computation using complex derivatives for functions of $z$ and $\bar z$ and writing for a scalar function $f$ the formulas
\[
\nabla f = \dfrac{\partial f}{\partial z}\begin{pmatrix}1\\i\end{pmatrix} + 
    \dfrac{\partial f}{\partial \bar z}\begin{pmatrix}1\\-i\end{pmatrix},\quad 
\div\begin{pmatrix}f\\i f\end{pmatrix}=2\dfrac{\partial f}{\partial \bar z},\quad 
\div\begin{pmatrix}f\\-i f\end{pmatrix}=2\dfrac{\partial f}{\partial z}\,,
\]
and
\[
   \Delta f=4\dfrac{\partial ^2f}{\partial z\,\partial  \bar z}\,.
\]
In this way we find that for any $\lambda\in\C\setminus\{0,1\}$, the following four vectors $\bw^k_\lambda$ given by 
\[
   \bw^1_\lambda = \bw_{(1)},\quad
   \bw^2_\lambda = \widetilde\bw_{(2)},\quad 
   \bw^3_\lambda = \bw_{(3)} + \frac{2\sigma-1}{\lambda}\, \widetilde\bw_{(1)},\quad
   \bw^4_\lambda = \widetilde\bw_{(4)} + \frac{2\sigma-1}{\lambda}\,  \bw_{(2)},
\]
form a basis\footnote{To obtain a basis valid for $\lambda=0$, we could define $\bw^3_\lambda$ by 
$\bw_{(3)}+\widetilde\bw_{(4)} + \frac{2\sigma-1}{\lambda}\, (\widetilde\bw_{(1)}+\bw_{(2)} -\bw_{(1)}-\widetilde\bw_{(2)})$ and $\bw^4_\lambda$ by 
$\bw_{(3)}-\widetilde\bw_{(4)} + \frac{2\sigma-1}{\lambda}\, (\widetilde\bw_{(1)}-\bw_{(2)} +\bw_{(1)}-\widetilde\bw_{(2)})$.} for the space $\gW_\sigma(\lambda)$.

\smallskip
{\sc Step 2.} We look for conditions on $\lambda$ so that there exists a nonzero $\bw\in\gW_\sigma(\lambda)$ which satisfies the Dirichlet conditions on $\theta=\pm\frac\omega2$. We note that, setting
\[
   a=\frac{\sin(\lambda {-}1)\omega}{\sin(\lambda \omega)}\quad \mbox{and}\quad
   b=\frac{\sin\omega}{\sin(\lambda \omega)}\,,
\]
for any $\varepsilon\in\R$ the function
\[
   \bw = \bw_{(3)} - a \bw_{(1)} - b \bw_{(2)} + 
   \varepsilon \left( \widetilde\bw_{(4)} - a \widetilde\bw_{(2)} - b \widetilde\bw_{(1)}\right)
\]
is zero on $\theta=\pm\frac\omega2$. It is clear that $\bw$ belongs to $\gW_\sigma(\lambda)$ iff
$\varepsilon=\pm1$ and $-\varepsilon b = \frac{2\sigma-1}{\lambda}$, i.e. 
\begin{equation}
\label{E:sigma}
   \sigma = \frac12\Big(1-\varepsilon\frac{\lambda\sin\omega}{\sin(\lambda \omega)}\Big),\quad
   \varepsilon=\pm1.
\end{equation}
Note that this is the same equation as \eqref{E:char}.
The associated singular function is
\begin{equation}
\label{E:wlam}
   \bw_\lambda = \bw^3_\lambda + \varepsilon \bw^4_\lambda - 
   a (\bw^1_\lambda + \varepsilon\bw^2_\lambda).
\end{equation}

\subsection{Singular sequences at corners and essential spectrum}\label{Ss:SingSeq}

The result of the Kondrat'ev theory for our polygon can be written as follows.

\begin{theorem}
\label{L:charpolyg}
Let $\Omega$ be a polygon in the plane. Assume that $\sigma$ does not belong to $\{0,\frac12,1\}$.
Then the Cosserat operator $L_{\sigma}$ is Fredholm from $H^{1}_{0}(\Omega)^2$ to $H^{-1}(\Omega)^2$ if and only if for each corner $\bc\in\gC$ and $\omega=\omega_{\bc}$ the characteristic equations \eqref{E:char} have no solution on the line $\Re\lambda=0$.
\end{theorem}

The ''if'' part is a consequence of Theorem \ref{T:Mellin}. For the ``only if'' part, we could simply quote \cite[Remark 6.3.4]{KozlovMazyaRossmann97b}, but we prefer to present a rather simple explicit proof in the spirit of \cite[\S9.D]{Dauge88}, namely the  construction of a singular sequence approximating the non-$H^{1}$ corner singularity.

\medskip
Choose a corner $\bc$, which we assume to sit at the origin, set $\omega_\bc=\omega$. Let $R>0$ be such that $\Omega$ coincides with the plane sector $\Gamma_\bc$ in the ball of center $\bc$ and radius $2R$.
Now choose a cut-off function $\chi\in \sC^1(\R^{2})$ such that $\chi(x)=1$ for $|x|\leq R$ and $\chi(x)=0$ for $|x|>2R$, and for any complex $\lambda$ set $\bu_\lambda =\chi\,\bw_\lambda$ with $\bw_\lambda$ the singular function defined in \eqref{E:wlam}. 
Then $\bu_\lambda=0$ on $\partial\Omega$. Therefore $\bu_\lambda\in H^1_0(\Omega)^2$ if and only if $\Re \lambda >0$. 

\smallskip
Let $\sigma$ be such that the equation \eqref{E:char} has a solution $\lambda\in i\R$, and take $\varepsilon\in\{\pm1\}$ so that $\sigma =\frac12\,(1{-}\varepsilon\, \frac{\lambda \sin\omega}{\sin(\lambda \omega)})$. Define
\begin{equation}
\label{E:lamn}
   \lambda_n=\lambda+\frac1n\quad\mbox{and}\quad
   \sigma_n =
   \frac12\,\Big(1{-}\varepsilon\, \frac{\lambda_n \sin\omega}{\sin(\lambda_n \omega)}\Big).
\end{equation}
Then
\begin{equation}
\label{E:Lsigmau}
   L_{\sigma}\bu_{\lambda_n}=(\sigma {-}\sigma_{n})\Delta \bu_{\lambda _n}+\bff_n 
   \quad\mbox{ with }\quad
   \bff_n = L_{\sigma_{n}}\bu_{\lambda_n}\,.
\end{equation}
From $L_{\sigma_{n}}\bw_{\lambda_n}=0\,$ follows that $\bff_{n}=0$ for $|x|\not\in[R,2R]$ and that $\|\bff_n\|_{0,\Omega}$ can be estimated by 
$\|\bw_{\lambda_{n}}\|_{H^{1}(\{R<|x|<2R\})}$, which remains bounded as $n\to\infty$. 
Note that for $n\to\infty$, $\sigma_n {-}\sigma\to 0$ and 
$\|\bu_{\lambda_n}\|_{0,\Omega}$ remains bounded, while $|\bu_{\lambda_n}|_{1,\Omega}\to\infty$.
Using the fact that $\Delta$ is an isometry from $H^{1}_{0}(\Omega)$ to $H^{-1}(\Omega)$, we obtain with \eqref{E:Lsigmau}
\begin{equation}
\label{E:lim0}
\lim_{n\to\infty} \frac{|L_{\sigma}\bu_{\lambda_n}|_{-1,\Omega}}
{|\bu_{\lambda _n}|_{1,\Omega}}=0.
\end{equation}
Altogether, this shows that $L_{\sigma}$ cannot satisfy an a-priori estimate of the form
\[
 |\bu|_{1,\Omega} \le \alpha |L_{\sigma}\bu|_{-1,\Omega} + \beta \|\bu\|_{0,\Omega}
 \quad\mbox{ with constants }\alpha,\beta\ge0.
\]
Hence $L_{\sigma}$ is not Fredholm and $\sigma$ belongs to the essential spectrum.

\bigskip
We can now conclude in the case of polygonal domains.

\begin{theorem}
\label{T:esspolyg}
Let $\Omega\subset\R^{2}$ be a polygon with corner angles $\omega_{\bc}\in(0,2\pi)$, $\bc\in\gC$. Then the essential spectrum of the Cosserat problem in $\Omega$ is given by
\begin{equation}
\label{E:esspolyg}
 \sigma_{\mathrm{ess}}(\SS) = \{1\}\cup 
   \bigcup_{\bc\in\gC}
   \Big[\frac12 - \frac{\sin\omega_{\bc}}{2\omega_{\bc}}, \frac12 + \frac{\sin\omega_{\bc}}{2\omega_{\bc}}\Big]\;.
\end{equation} 
The LBB constant satisfies 
\begin{equation}
\label{E:LBBpolyg}
 0<\beta(\Omega) \le\, \min_{\bc\in\gC} 
  \Big(\frac12 - \frac{\sin\omega_{\bc}}{2\omega_{\bc}}\Big)^{\frac12} \;.
\end{equation}
\end{theorem}
\begin{proof}
As $\Omega$ is Lipschitz, $\beta(\Omega)>0$ \cite{HorganPayne1983}. 
For $\sigma_{\mathrm{ess}}(\SS)$, 
we have to find the $\sigma$ for which \eqref{E:char} has solutions $\lambda$ with $\Re\lambda=0$. The function $\lambda\mapsto\frac{\sin\lambda\omega}{\lambda\omega}$ maps $i\R$ to $[1,\infty)$. For $\sigma\ne\frac12$, the necessary and sufficient condition is therefore
\[
 -\frac{\sin\omega}\omega \le 1-2\sigma \le \frac{\sin\omega}\omega\,.
\]
This proves \eqref{E:esspolyg}.
\end{proof}

\begin{remark}
\label{R:curvedpol}
The result of Theorem \ref{T:esspolyg} remains true if $\Omega$ is a curved polygon, that is a bounded Lipschitz domain with a piecewise $\sC^{2}$ boundary. This follows from general perturbation techniques that are part of the Kondrat'ev theory \cite[\S2]{Kondratev67}, see also \cite[Ch.~6]{KozlovMazyaRossmann97b}. If the corner angles have cusps, Theorem \ref{T:esspolyg} extends as follows: An inward cusp or crack ($\omega=2\pi$) does not contribute to the nontrivial essential spectrum \cite{Friedrichs1937}, whereas in the presence of an outward cusp ($\omega=0$), the essential spectrum of the Cosserat problem fills the entire interval $[0,1]$ \cite{Friedrichs1937,Tartar_NS2006}.
The estimate~\eqref{E:LBBpolyg} for $\beta(\Omega)$ was proved by Stoyan in \cite{Stoyan1999}, using the equivalence between the Friedrichs constant and the LBB constant \cite{HorganPayne1983,CoDa_IBAFHP}.
\end{remark}

\section{Three-dimensional domains with edges or corners}\label{S:3D}

Let $\Omega\subset\R^3$. We consider domains with edges and corners. For the case when $\Omega$ contains edges in its boundary, we prove that the essential spectrum of the Cosserat problem contains an interval that depends on the opening angle of the edges. For the case when $\Omega$ has a conical point with an axisymmetric tangent cone, we describe how to reduce the question to the determination of the roots of a sequence of holomorphic functions and we show the result of numerical computations. Polyhedral corners are the subject of some comments at the end of this section.

\subsection{Edges}\label{Ss:Edges}
In a first step, we assume that the boundary of $\Omega$ contains a piece of a straight edge of opening $\omega\in(0,2\pi)$ in the sense that for some $R>0$
\begin{multline}
\label{E:edge}
   \Big\{x\in \Omega \mid \sqrt{x_1^2{+}x_2^2}\leq 2R,\ |x_3|\leq 2R\Big\} \\ =
   \Big\{x\in \R^3 \mid \sqrt{x_1^2{+}x_2^2}\leq 2R,\ |x_3|\leq 2R,\ 
   |\arg(x_1{+}ix_2)|<\frac\omega2\Big\}.
\end{multline}

\begin{theorem}
\label{T:essedge}
 Under the hypothesis \eqref{E:edge}, the interval 
 $[\frac12-\frac{\sin\omega}{2\omega},\frac12+\frac{\sin\omega}{2\omega}]$ is contained in the essential spectrum of the Cosserat problem.
\end{theorem}

\begin{proof}
Let $\sigma\in[\frac12-\frac{\sin\omega}{2\omega},\frac12+\frac{\sin\omega}{2\omega}]\setminus\{\frac12\}$. We construct a singular sequence based on the one that we have used for a two-dimensional corner in Section~\ref{Ss:SingSeq}: Like there, we set $z=x_{1}+ix_{2}$, and for $\lambda\in\C$ we define  $\bw_{\lambda}=(w_{\lambda ,1},w_{\lambda ,2})^\top$  as in equation \eqref{E:wlam} and use the same cut-off function $\chi$. 
Now we choose another cut-off function $\theta \in \sC^2(\R)$ that satisfies 
$\theta (x_3)=1$ if $|x_3|\leq R$ and $\theta (x_3)=0$ if $|x_3|\geq  2R$.
Then we define $\bu_{\lambda}$ as
\begin{equation*}
\bu_\lambda(x)=\theta (x_3)\,\chi(x_1,x_2)
\begin{pmatrix}
 w_{\lambda ,1}(z)\\ w_{\lambda ,2}(z)\\ 0
\end{pmatrix}.
\end{equation*}
For $\lambda$, $\lambda_{n}$, $\sigma_{n}$ chosen as before, see \eqref{E:lamn}, the vector functions $\bu_{\lambda _n}$ belong to $H^{1}_{0}(\Omega)^{3}$ and $L_{\sigma}\bu_{\lambda _n}$ has still the expression \eqref{E:Lsigmau} with $\bff_n = (f_{n,1},f_{n,2},f_{n,3})^\top$ given by
\begin{align*}
 \begin{pmatrix}
   f_{n ,1}\\f_{n,2}
 \end{pmatrix} &=
   \theta(x_3)\, L^{2\mathrm{D}}_{\sigma_{n}}(\chi \bw_{\lambda_n}) 
   + \sigma_n\theta ''(x_3)\,\chi \bw_{\lambda_n}\\
 f_{n,3} &= -\,\theta'(x_3)\div (\chi\bw_{\lambda_n}) \,.
\end{align*}
We conclude as before that $\|\bu_{\lambda_n}\|_{0,\Omega}$ and 
$\|\binom{f_{n ,1}}{f_{n ,2}}\|_{0,\Omega}$ remain bounded as $n\to\infty$ whereas $|\bu_{\lambda_n}|_{1,\Omega}\to\infty$.
 As for $f_{n,3}$, we see that $|f_{n,3}|_{-1,\Omega}$ remains bounded, so that we can again conclude that $\bu_{\lambda_n}$ is a singular sequence satisfying \eqref{E:lim0},
and that $\sigma$ therefore belongs to the essential spectrum.
\end{proof}

\begin{remark}
\label{R:edge} 
Yet another contribution to the essential spectrum may come from the edges, this time not from the essential spectrum of the transversal Cosserat problem $L^{2\mathrm{D}}_{\sigma}$ as described in Theorem~\ref{T:essedge}, but from \emph{eigenvalues} of the ``edge symbol'', which is the  boundary value problem defined on the sector $\Gamma= \{(x_1,x_2)\in \R^2 \mid 
|\arg(x_1{+}ix_2)|<\frac\omega2\}$, cf.\ \eqref{E:edge},
\begin{equation}
\label{E:edgesymb}
   \R\ni\xi \mapsto L_\sigma(\xi) = \sigma(\partial_1^2+\partial_2^2-\xi^2)\mathbb{I}_3 -
   \begin{pmatrix}
   \partial^2_1 & \partial_1\partial_2 & i\xi\partial_1 \\
   \partial_1\partial_2 & \partial^2_2 & i\xi\partial_2 \\
   i\xi\partial_1 & i\xi\partial_2 & -\xi^2 \\
   \end{pmatrix} : H^1_0(\Gamma)^3 \to H^{-1}(\Gamma)^3,
\end{equation}
obtained by partial Fourier transformation along the edge. Indeed, it is known that for an elliptic boundary value problem on a domain with edges to be Fredholm, it is not sufficient that its edge symbol is Fredholm (this latter condition is satisfied as soon as the transversal Mellin symbol is invertible on a certain line $\Re\lambda=\mathrm{const}$), but it has to be \emph{invertible} for all $\xi\neq0$ \cite[Theorem~10.1]{MazyaPlamenevskii80b}, see also \cite{MazyaRossmann10}. For our Cosserat problem, if such eigenvalues exist below the essential spectrum of the transversal problem $L^{2\mathrm{D}}_{\sigma}$, then the lower bound of the essential spectrum on $\Omega$, and therefore the LBB constant, may be smaller than what is described by Theorem~\ref{T:essedge}. Whether this actually happens is unknown, and some numerical experiments we did rather seem to indicate that it does not.
\end{remark}
\Bk

\begin{remark}
\label{R:curvededge}
The result of Theorem \ref{T:essedge} remains true for some curved edges, too, similar to what we mentioned in two dimensions in Remark~\ref{R:curvedpol}. In particular, if $\Omega$ is a finite straight cylinder with a smooth base, then the interval corresponding to the angle $\omega=\frac\pi2$, namely $[\frac12-\frac1\pi, \frac12+\frac1\pi]$, will belong to the essential spectrum. 
\end{remark}

\subsection{Conical points}\label{Ss:Cones}
We assume now that $\Omega$ is a domain in $\R^3$ with conical points. Theorem \ref{T:Mellin}  applies with the critical abscissa $\Re\lambda=-\frac12$. We can prove exactly like in dimension 2 that the sufficient condition for Fredholmness is also necessary. Now the question is: Does a 3D cone produce a full interval of nonzero length inside the essential spectrum? Do we have symmetry with respect to $\sigma=\frac12$ like in 2D? We do not have a general answer to the first question, but a partial answer in the case of axisymmetric cones. This example shows that the symmetry property with respect to $\sigma=\frac12$ is lost.

We learn from \cite{KozlovMazyaSchwab91,KozlovMazyaSchwab92} that semi-analytical calculation of the spectrum of the Mellin symbol $\gA_\sigma$ associated with an axisymmetric cone $\Gamma$ is made possible by the use of the Boussinesq representation and separation of variables in spherical coordinates. Then, relying on the equivalence \eqref{E:equiv}, we can follow the same track as in 2D, looking for homogeneous solutions $\bw$ of the equation
\begin{equation}
\label{E:homo}
   (\sigma\Delta - \nabla\div)\bw=0
\end{equation}
without boundary conditions in a first step, and imposing the homogeneous Dirichlet conditions in a second step.

We give a rapid overview of this procedure (see \cite[\S3.7]{KozlovMazyaRossmann01} for details). The Boussinesq representation for solutions of the equation \eqref{E:homo} states that $\bw$ can be found as a linear combination of the three following particular solutions of the same equation
\begin{equation}
\label{E:Boussinesq}
   \bw_1 = \nabla \Psi,\quad
   \bw_2 = \curl(\Theta\,\vec e_3),\quad
   \bw_3 = \nabla (x_3 \Lambda) + 2(\sigma-1)\Lambda\vec e_3,
\end{equation}
where the scalar functions $\Psi$, $\Theta$, and $\Lambda$ are harmonic functions. Here we choose $\vec e_3$ as the director of the cone axis. Then homogeneous $\bw$'s of degree $\lambda$ are given by finding $\Psi$ and $\Theta$ homogeneous of degree $\lambda+1$, and $\Lambda$ of degree $\lambda$. Using spherical coordinates $(r,\theta,\varphi)$ such that
\[
   x_1 = r \sin\theta \cos\varphi,\quad	
   x_2 = r \sin\theta \sin\varphi,\quad	
   x_3 = r \cos\theta
\]
we split the 3D problem into an infinite sequence of 2D problems in $(r,\theta)$ parametrized by the azimuthal frequency $m\in\Z$. The harmonic generating function can be written in separated variables as 
\begin{equation}
\label{E:harm}
   \begin{gathered}
   \Psi = r^{\lambda+1} P^{-m}_{\lambda+1}(\cos\theta) \cos(m\varphi),\quad
   \Theta = r^{\lambda+1} P^{-m}_{\lambda+1}(\cos\theta) \sin(m\varphi),\\
   \mbox{and}\quad \Lambda = r^{\lambda} P^{-m}_{\lambda}(\cos\theta) \sin(m\varphi),
\end{gathered}
\end{equation}
where $P^\mu_\nu$ is the associated Legendre function of the first kind of order $\mu$ and degree $\nu$. Combining \eqref{E:Boussinesq} and \eqref{E:harm}, we find for each $m\in\Z$ three independent solutions $\bw^{m,1}_\lambda$, $\bw^{m,2}_\lambda$, and $\bw^{m,3}_\lambda$ of degree $\lambda$ for equation \eqref{E:homo}. 

\begin{figure}
\centerline{\includegraphics[scale=0.41]{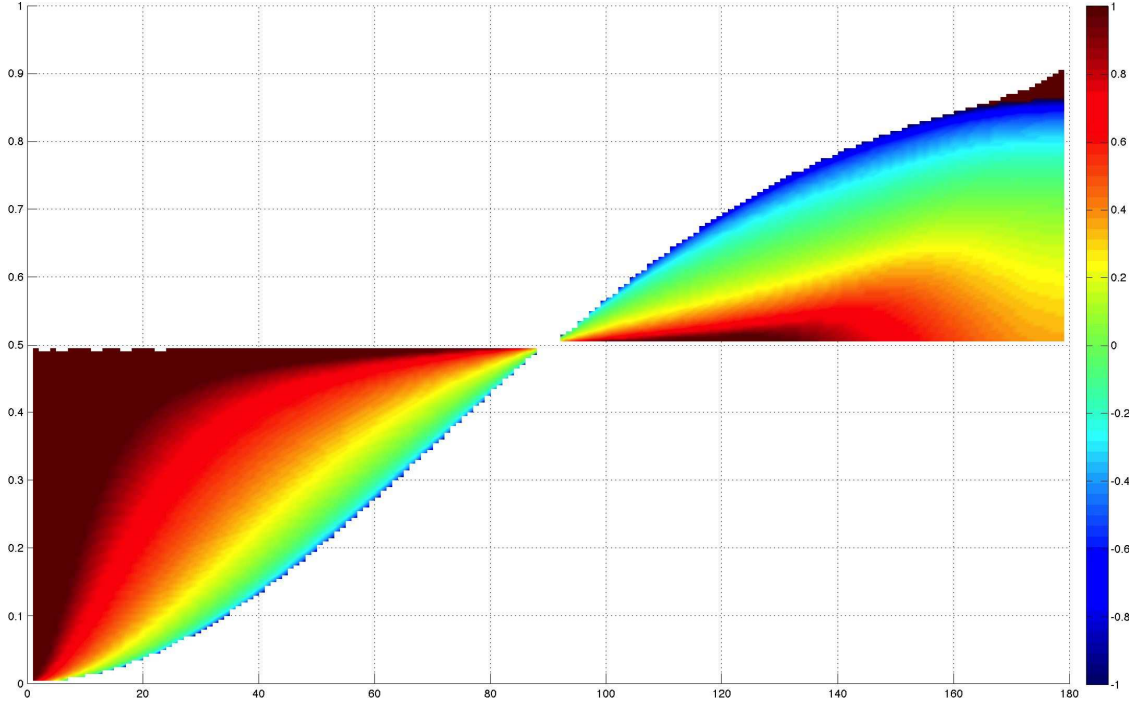}}
\caption{$m=0$: Color plot of the decimal logarithm of  $\Im\lambda$ of roots $\lambda$ with real part $-\frac12$ of characteristic equation $\gM^0_\sigma(\lambda)=0$ as a function of the opening $\omega$ (in degrees, abscissa) and the parameter $\sigma\in[0,1]$ (ordinate).}
\label{F4.1}
\end{figure}

Denote by $\omega$ the opening of the cone $\Gamma$, which means that $\Gamma$ is defined by the condition $\theta\in[0,\omega)$ in spherical coordinates. The exponents $\lambda$ we are looking for are those for which there exists an integer $m$ such that the following $3\times3$ matrix is singular for all $r>0$ and $\varphi$
\[
   \begin{pmatrix}
   w^{m,1}_{\lambda,1}(r,\omega,\varphi) & 
   w^{m,2}_{\lambda,1}(r,\omega,\varphi) & 
   w^{m,3}_{\lambda,1}(r,\omega,\varphi) \\
   w^{m,1}_{\lambda,2}(r,\omega,\varphi) & 
   w^{m,2}_{\lambda,2}(r,\omega,\varphi) & 
   w^{m,3}_{\lambda,2}(r,\omega,\varphi) \\
   w^{m,1}_{\lambda,3}(r,\omega,\varphi) & 
   w^{m,2}_{\lambda,3}(r,\omega,\varphi) & 
   w^{m,3}_{\lambda,3}(r,\omega,\varphi) 
   \end{pmatrix}
\]
Using the special forms \eqref{E:Boussinesq} and \eqref{E:harm}, the variables $r$ and $\varphi$ disappear and after some obvious simplification we are left with the matrix\footnote{The matrix \eqref{E:mat3D} that we reproduce here is equivalent to the one we find in the preprint \cite{KozlovMazyaSchwab91}. Its correctness can be checked. Unfortunately a misprint appeared in further references \cite{KozlovMazyaSchwab92,KozlovMazyaRossmann01}: The factor $-m$ in $(\lambda+1-m)$ is missing in the last term of the second column. Nevertheless, the numerical computations presented in these latter references where made with the correct formulas.}

\begin{figure}
\centerline{\includegraphics[scale=0.41]{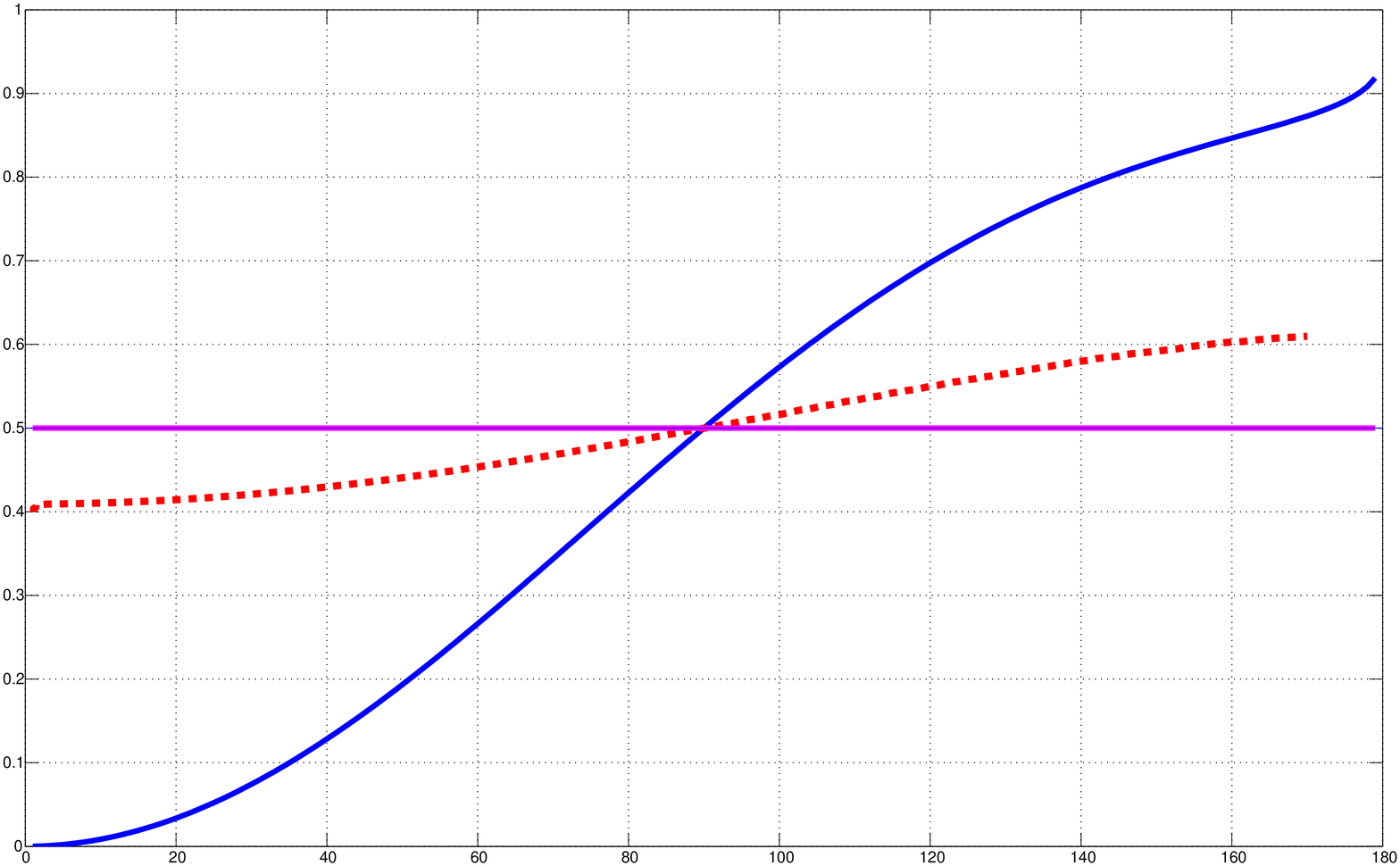}}
\caption{$m=0,1$: Boundary of regions $\gR^0$ (solid line) and $\gR^1$ (dashed line) in the plane $(\omega,\sigma)$: Opening $\omega$ (in degrees, abscissa) and $\sigma\in[0,1]$ (ordinate).}
\label{F4.2}
\end{figure}

{\small
\begin{equation}
\label{E:mat3D}
   \begin{pmatrix}
   (\lambda+1)P^{-m}_{\lambda+1}(\cos\omega) & 
   m P^{-m}_{\lambda+1}(\cos\omega) &
   (\lambda+2\sigma-1) \cos\omega \,P^{-m}_{\lambda}(\cos\omega) 
   \\[2ex]
   \begin{minipage}{0.3\textwidth}
   $(\lambda+1) \cos\omega P^{-m}_{\lambda+1}(\cos\omega)$ \\
   $-  (\lambda+1-m)P^{-m}_{\lambda}(\cos\omega)$
   \end{minipage} & 
   m\cos\omega P^{-m}_{\lambda+1}(\cos\omega) & 
   \begin{minipage}{0.3\textwidth}
   $(\lambda+1+m)\cos\omega \,P^{-m}_{\lambda+1}(\cos\omega)$\\
   $+(1-2\sigma)\sin^2\omega \,P^{-m}_{\lambda}(\cos\omega)$ \\
   $- (\lambda+1)\cos^2\omega \,P^{-m}_{\lambda}(\cos\omega)$
   \end{minipage} 
   \\[2ex]
   -m P^{-m}_{\lambda+1}(\cos\omega) &
   \begin{minipage}{0.3\textwidth}
   $(\lambda+1-m)\cos\omega \,P^{-m}_{\lambda}(\cos\omega)$ \\
   $-(\lambda+1)\,P^{-m}_{\lambda+1}(\cos\omega)$
   \end{minipage} &
   -m\cos\omega\,P^{-m}_{\lambda}(\cos\omega)
   \end{pmatrix}
\end{equation}
}

With this matrix at hand, it is possible to compute its determinant $\gM^m_\sigma(\lambda)$ and find, for any chosen $m$, couples $(\omega,\sigma)$ for which the equation $\gM^m_\sigma(\lambda)=0$ has roots on the line $\Re\lambda=-\frac12$. We present in Fig.~\ref{F4.1} the region $\gR^0$ of the $(\omega,\sigma)$ plane where such roots can be found when $m=0$. This region is clearly non-symmetric with respect to $\sigma=\frac12$.

The region $\gR^1$ that we have calculated for $m=1$ (which is the same for $m=-1$) is strictly contained in the region $\gR^0$ associated with $m=0$, see Fig.~\ref{F4.2}. We observe that when $|m|$ is increasing, the region $\gR^m$ is shrinking.

\subsection{Polyhedra}\label{R:polyhedra}
By polyhedron we understand a bounded Lipschitz domain $\Omega\subset\R^{3}$ the boundary of which consists of a finite set of plane polygons, the faces. This set being chosen in a minimal way, the edges $\be$ of $\Omega$ are the segments which form the boundaries of the faces, and the corners $\bc$ are the corners of the faces. Let $\gE$ and $\gC$ be the sets of edges and corners, respectively. To each edge $\be$ is associated a transversal sector $\Gamma_\be$ and its opening $\omega_\be$. To each corner $\bc$ is associated the tangent infinite polyhedral cone $\Gamma_\bc$ and its section $G_\bc=\Gamma_\bc\cap\Sbb^2$.

Let $\sigma\in\C$ be the Cosserat spectral parameter. For each edge $\be$, the edge symbol  $L^\be_\sigma(\xi)$ is defined by \eqref{E:edgesymb}. For each corner $\bc$, the corner Mellin symbol $\gA^\bc_\sigma(\lambda)$ is defined by \eqref{E:MellinS}. If $\sigma\not\in\{0,\frac12,1\}$, and if the following two conditions are satisfied
\begin{gather}
\label{E:edgeinv}
   \forall\be\in\gE,\quad L^\be_\sigma(\xi) \ \ \mbox{is invertible for all}\ \  \xi\neq0 \\
\label{E:cornerinv}
   \forall\bc\in\gC,\quad \gA^\bc_\sigma(\lambda) \ \ \mbox{is invertible for all}\ \  \lambda,\,
   \Re\lambda=-\tfrac12.
\end{gather}
then $L_\sigma$ is Fredholm on $\Omega$. These conditions are also necessary and determine the essential spectrum of the operator $\SS$ \eqref{E:SS}. We recall that, cf.~Remark \ref{R:edge}, the condition
\[
   \sigma\not\in\Big[\frac12 - \frac{\sin\omega_{\be}}{2\omega_{\be}}, \frac12 + \frac{\sin\omega_{\be}}{2\omega_{\be}}\Big]
\]
is necessary for \eqref{E:edgeinv} to hold but is, possibly, not sufficient.

\section{Finite element computations for rectangles and cuboids}\label{S:Rectangles}

\subsection{Rectangles}
For any rectangle $\Omega\subset\R^{2}$, Theorem~\ref{T:esspolyg} shows that the essential spectrum of the Cosserat problem is the interval $[\frac12-\frac1\pi, \frac12+\frac1\pi]$. Therefore the Cosserat constant of $\Omega$ satisfies 
\begin{equation}
\label{E:sigmapi2}
 \sigma(\Omega) \le \frac12-\frac1\pi \sim 0.18169 \,.
\end{equation}
In contrast to the essential spectrum, the discrete spectrum depends on the actual shape of the rectangle, namely of its aspect ratio (but not on its width and height separately). Our convention is to characterize a rectangle by the number $a\in(0,1]$ such that its aspect ratio is $1\ee$:$\ee a$ or $a\ee$:$\ee1$. The square has its factor $a$ equal to $1$, and small $a$ corresponds to elongated rectangles (large aspect ratio).

To determine $\sigma(\Omega)$ and therefore the LBB constant, one can solve the Cosserat eigenvalue problem numerically and look for the smallest eigenvalue. A conforming discretization based on the Stokes eigenvalue problem~\eqref{E:Stokesp} consists in choosing a pair of finite dimensional spaces $\gU\subset H^1_0(\Omega)$ and $\gP\subset L^2(\Omega)$ and constructing the following discrete version of the Schur complement
\begin{equation}
\label{E:Schur}
   \SS(\gU,\gP) = B_1 R^{-1} B_1^\top + B_2 R^{-1} B_2^\top 
\end{equation}
where $R$ is the stiffness matrix associated with the $\nabla:\nabla$ bilinear form on $\gU\times \gU$ and $B_k$ is associated with the bilinear form $(u,p)\mapsto\int_\Omega \partial_k u\,p$ on $\gU\times\gP$, $k=1,2$. The discrete Cosserat eigenvalues $\tilde\sigma_j$, $j\ge1$, are the non-zero eigenvalues of $\SS(\gU,\gP)$ ordered increasingly. Two main difficulties are encountered:
\begin{enumerate}
\item[{\em (i)}] The discrete pair $(\gU,\gP)$ constructed from finite element spaces may have a behavior of its own, somewhat independent of the continuous pair $(H^1_0(\Omega),L^2(\Omega))$. This may give rise to spurious eigenvalues.
\item[{\em (ii)}] The regularity of the eigenvectors $p$ of $\SS$ depends on the eigenvalue $\sigma$ and gets worse as $\sigma$ is closer to the essential spectrum, cf Remark \ref{R:sings}. In  Figure \ref{F5.0} we plot the supremum of exponents $s$ such that $p$ belongs to $H^s(\Omega)$: This is the minimal $\Re\lambda$ for $\lambda$ solution of  $(1-2\sigma)\sin\frac{\lambda\pi}2 = \pm\lambda$ with positive real part, cf \eqref{E:reg}.
\end{enumerate}

\begin{figure}
\includegraphics[scale=0.45]{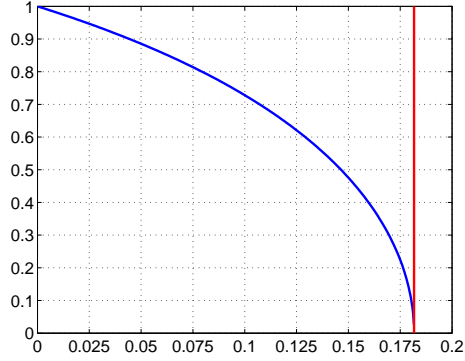}
\caption{Regularity exponent $s$ (ordinate) function of $\sigma$ (abscissa) for rectangles. The vertical asymptote is $\sigma=\frac12-\frac1\pi$.}
\label{F5.0}
\end{figure}

There exist estimates for $\sigma(\Omega)$ from above and from below which prove that $\sigma(\Omega)\to0$ like $\OO(a^2)$ as $a\to0$: 
\begin{equation}
\label{E:bounds}
   \sin^2\left(\tfrac12{\arctan a}\right) \ \le\ \sigma(\Omega) \ 
   \le\ 1-\frac{\sinh\rho}{\rho\cosh\rho},\quad \mbox{with}\quad \rho=\frac{a\pi}{2}.
\end{equation}
The lower bound is deduced from the Horgan-Payne estimate \cite{HorganPayne1983,CoDa_IBAFHP} and is equal to $\frac{a^2}{4}$ modulo $\OO(a^4)$. The upper bound is obtained by plugging the quasimode $p(x_1,x_2)=\cos x_1$ on the rectangle $(0,\pi)\times(-\rho,\rho)$ into the Rayleigh quotient of the Schur complement $\SS$ (see Lemma \ref{L:channel} below for a more general estimate of this type). It is an improvement at the order $\OO(a^4)$ of the upper bound $\frac{\pi^2a^2}{12}$ proven in \cite{ChizhOlsh00} by Chizhonkov and Olshanskii. Note that the upper bound in \eqref{E:bounds} proves that for any $a\le0.53127$, the bottom $\sigma(\Omega)$ of the spectrum of $\SS$ is an eigenvalue.

\renewcommand{\arraystretch}{1.25}
\begin{table}
\begin{tabular}{| c | c | c | c | }
\hline 
& Extrapolated value & Regularity exponent $s$ & Convergence rate ($n=5,6,7$) \\
\hline 
$\tilde\sigma_1$ & 0.031375609 & 0.93189 & 1.85618\\
$\tilde\sigma_2$ & 0.109538571 & 0.69036 & 1.37814 \\
\hline 
\end{tabular}
\medskip
\caption{Convergence rates for Cosserat eigenvalues on the rectangle $a=0.2$}
\label{T1}
\end{table}

The eigenfunctions are not very singular at the corners if the first eigenvalue is well below the minimum \eqref{E:sigmapi2} of the essential spectrum (Figure \ref{F5.0}), which is the situation for rectangles of large aspect ratio according to \eqref{E:bounds}. We have quantified this by evaluating convergence rates of the first and second eigenvalues when $a=0.2$ (aspect ratio 5\ee:\ee1), using uniform square meshes with $5\cdot2^n\cdot2^n$ elements ($n=1,\ldots,7$) and $\Q_2$-$\Q_1$ polynomial spaces, see Table \ref{T1} where a convergence rate equal to twice the regularity exponent is observed.
However, as the aspect ratio approaches $1$, the first Cosserat eigenvalue approaches the essential spectrum and the numerical results become less reliable.

We present in Figure~\ref{F5.1} computations done with the finite element library {\sl M\'elina++} with different choices of quadrilateral meshes (uniform 12$\times$12 for (a) and (c), strongly geometrically refined at corners with 144 elements for (b) and (d)) and different choices of polynomial degrees for $\bu$ and $p$ (tensor spaces of degree 8 and 6 for (a) and (b), 8 and 7 for (c) and (d)). Note that the meshes follow the elongation of the rectangles.

From these four discretizations, we observe a relative stability of the results below $\frac12-\frac1\pi$, if we except a quadruple eigenvalue that appears for the degrees (8,7) and is rather insensitive to the mesh and the aspect ratio. Moreover, with the same restrictions, the first eigenvalue sits between the explicit lower and upper bounds \eqref{E:bounds}, and further eigenvalues satisfy the estimate of \cite[Theorem 5]{Dobrowolski2005}. In contrast, the part of the numerical eigenvalues appearing above $\frac12-\frac1\pi$ is very sensitive to the mesh and flattens in a very spectacular way when a refined mesh is used.

As one can see from the graphs in Figure~\ref{F5.1}, at an aspect ratio of about 1\ee:\ee0.6 (golden ratio~?), the lowest computed eigenvalue crosses over into the interval occupied by the essential spectrum. This is observed rather stably in similar computations, and if true, it would mean that for all rectangles of smaller aspect ratio, in particular for the square, one would have the LBB constant corresponding to \eqref{E:sigmapi2}, as discussed in the introduction, see \eqref{E:betamaj}.
But due to the large numerical errors arising from spurious numerical eigenvalues and from the strong corner singularities near the essential spectrum of the continuous operator, the numerical evidence is not as convincing as one would wish.

\begin{figure}
\begin{tabular}{ccc}
\includegraphics[scale=0.45]{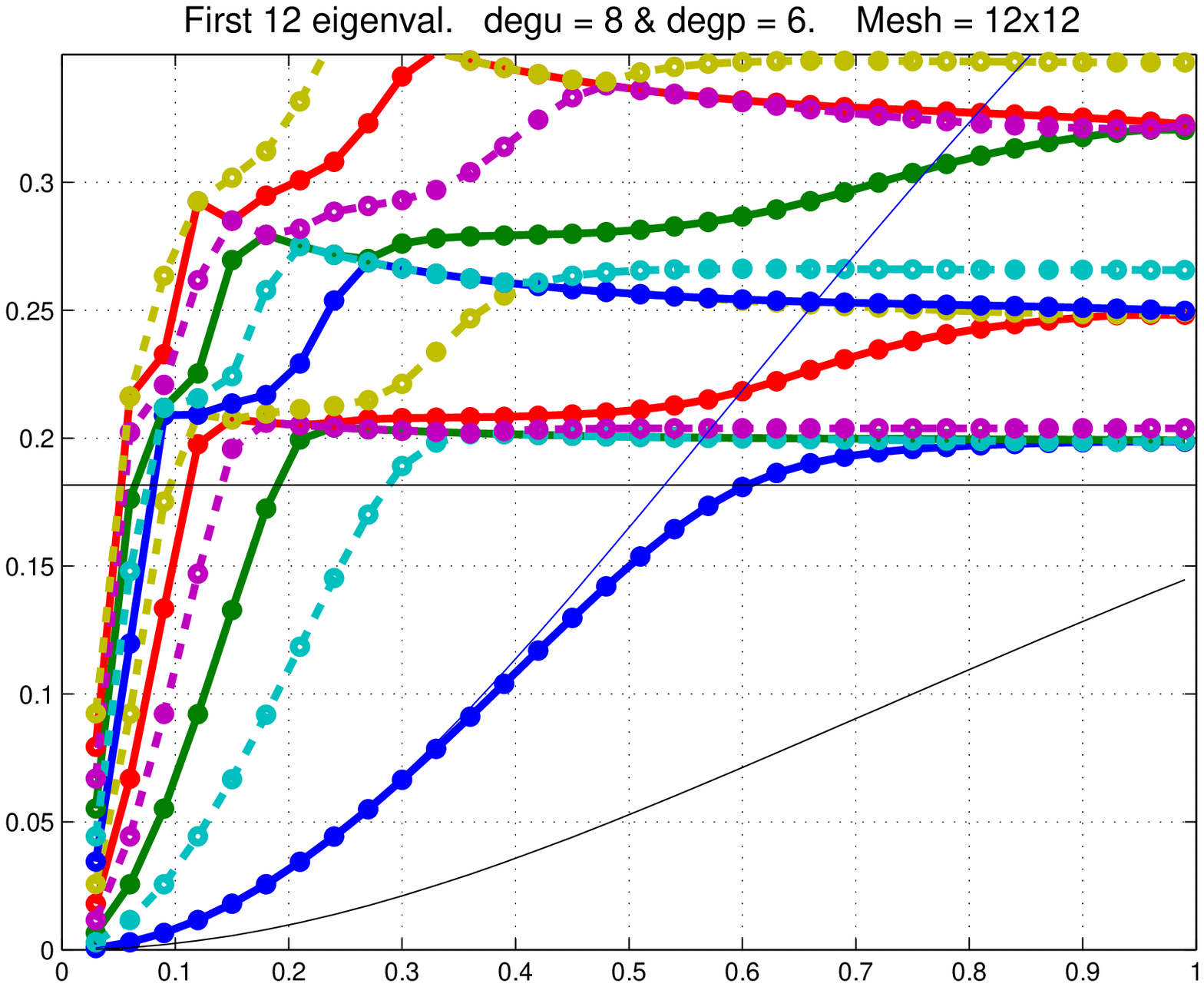} &
\includegraphics[scale=0.45]{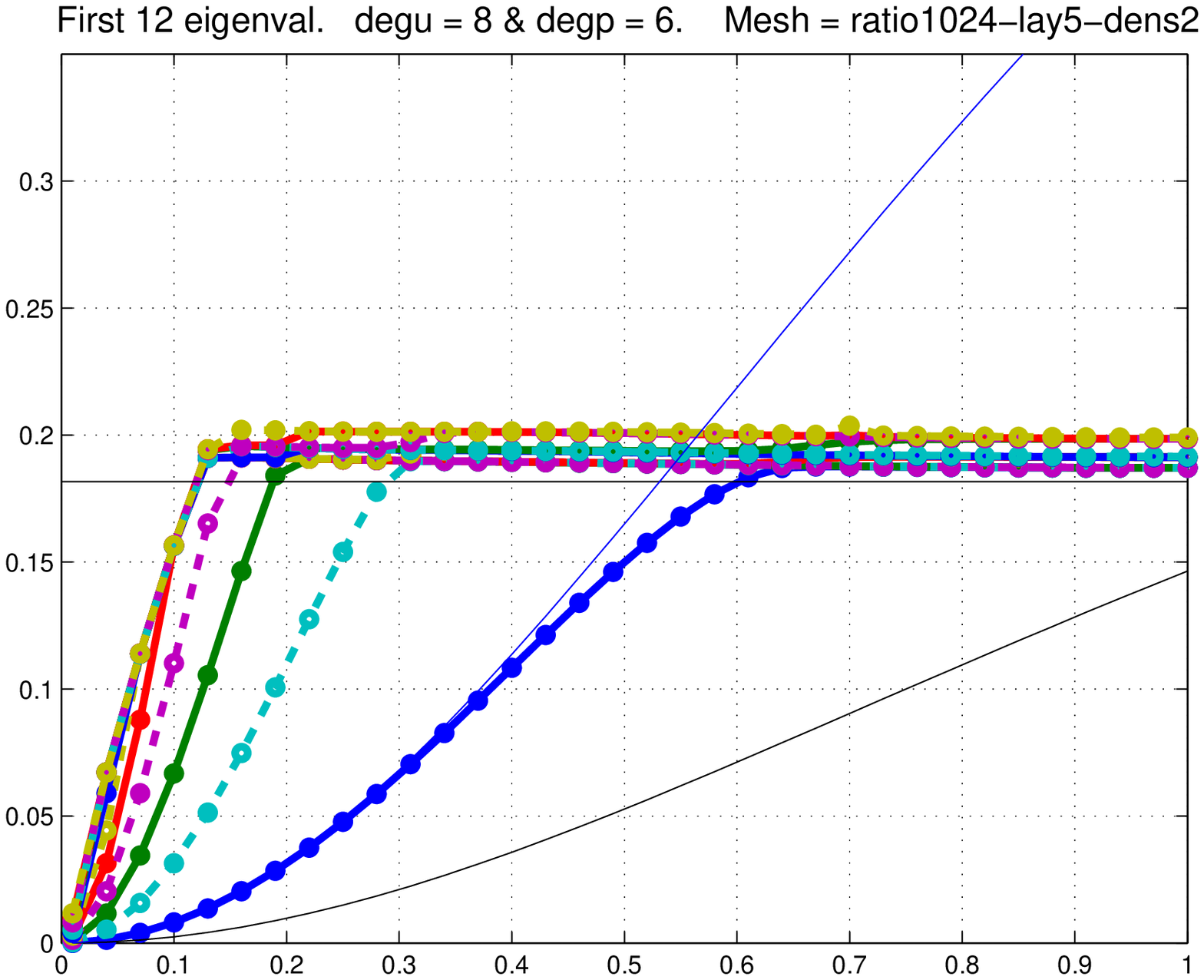} \\
(a) Uniform mesh, $\Q_8$ for $\bu$ and $\Q_6$ for $p$ & 
(b) Refined mesh, $\Q_8$ for $\bu$ and $\Q_6$ for $p$  \\
\includegraphics[scale=0.45]{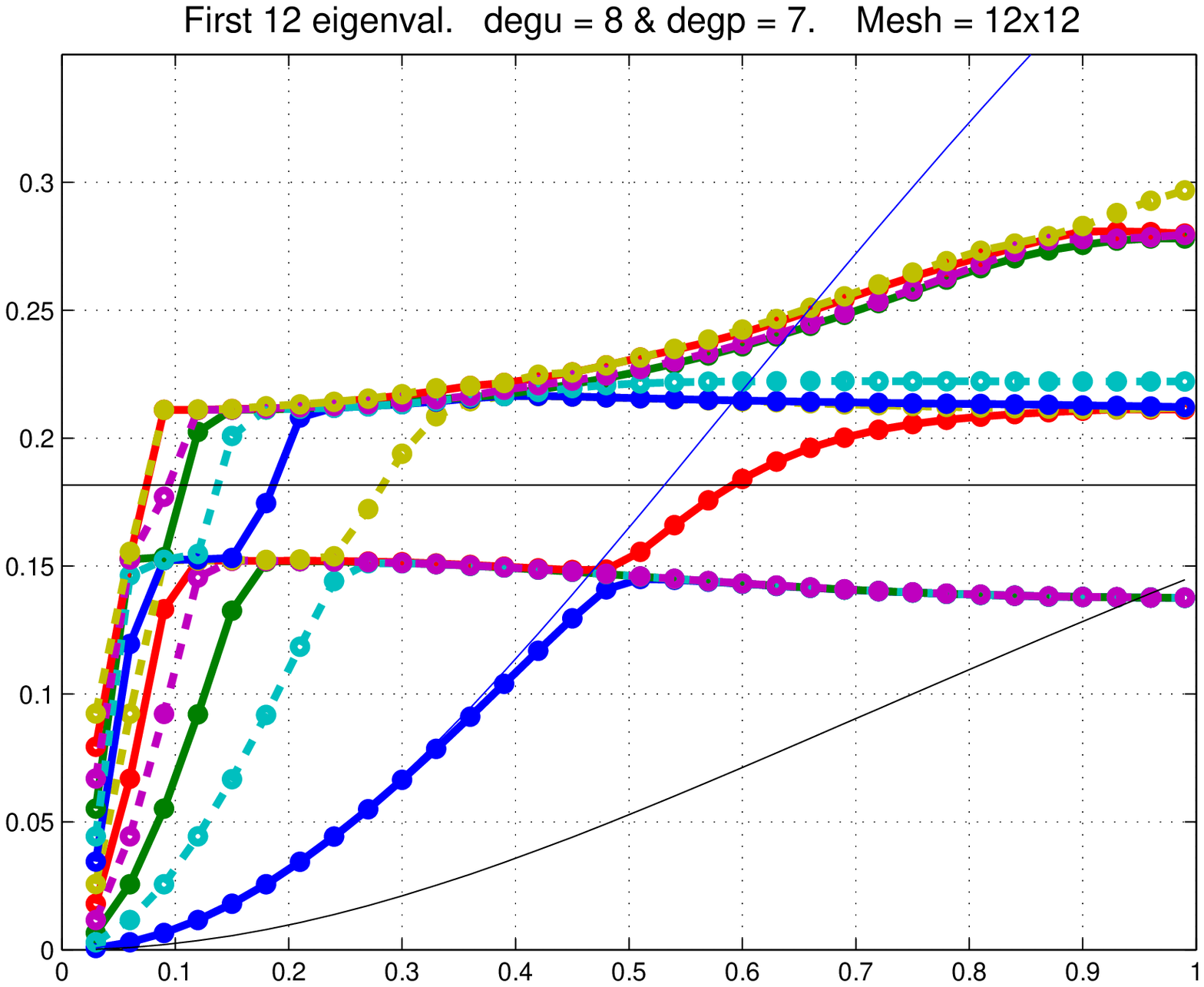} &
\includegraphics[scale=0.45]{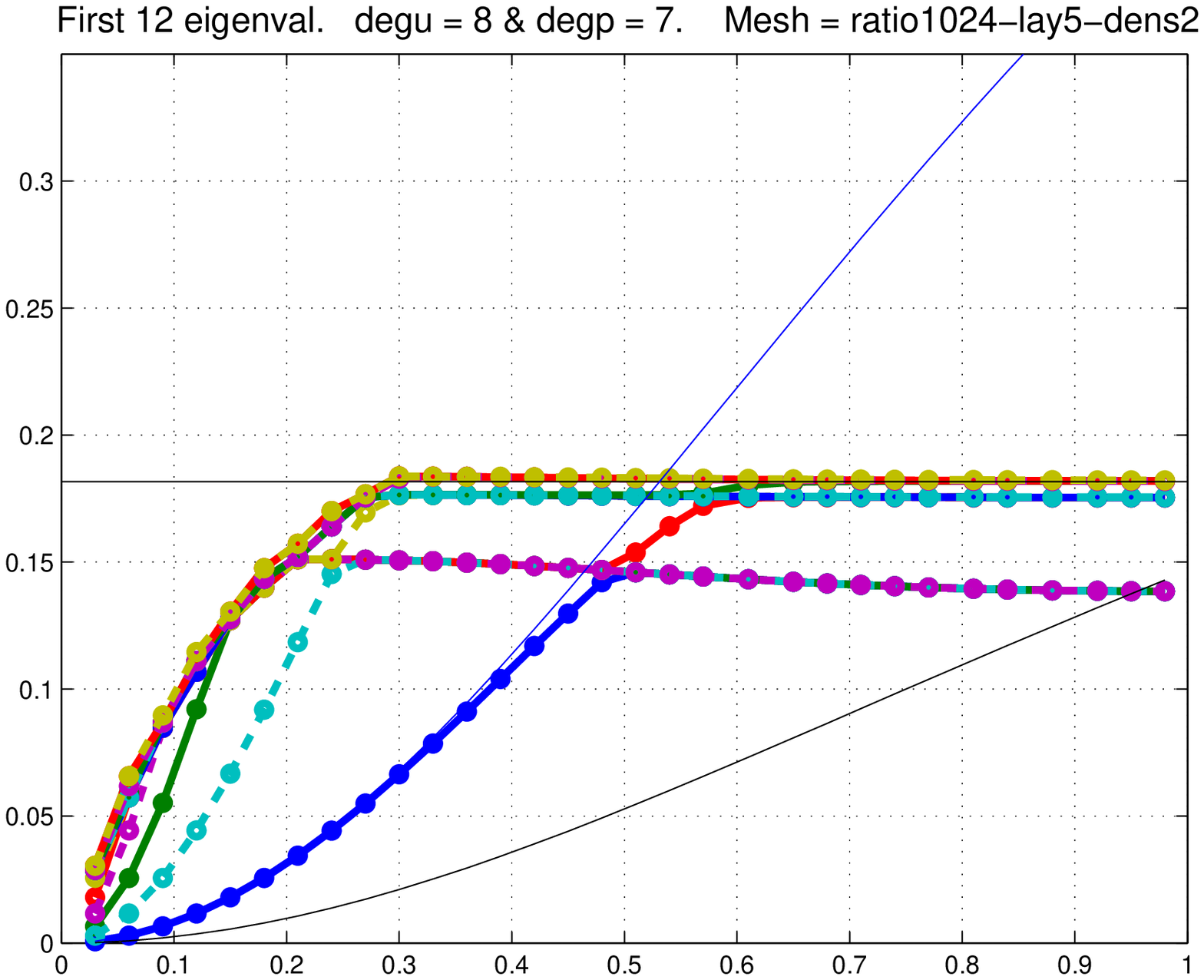} \\
(c) Uniform mesh, $\Q_8$ for $\bu$ and $\Q_7$ for $p$ & 
(d) Refined mesh, $\Q_8$ for $\bu$ and $\Q_7$ for $p$ 
\end{tabular}
\caption{First 12 computed Cosserat eigenvalues on rectangles vs parameter $a\in(0,1]$ (abscissa). The solid horizontal line is the infimum $\frac12-\frac1\pi$ of the essential spectrum. Solid curves are the lower and upper bounds \eqref{E:bounds}.}
\label{F5.1}
\end{figure}

\begin{figure}
\begin{tabular}{ccc}
\includegraphics[scale=0.45]{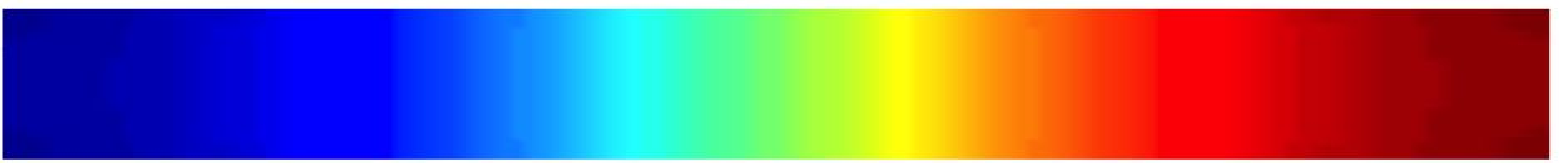} &
\includegraphics[scale=0.45]{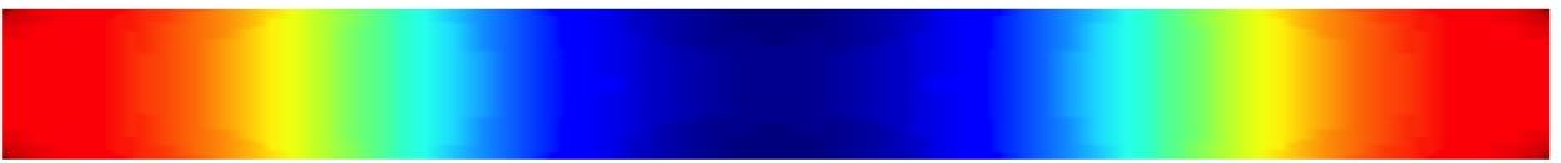} \\
$\sigma_1\simeq 0.008129$ & $\sigma_2\simeq 0.031410$ \\
\includegraphics[scale=0.45]{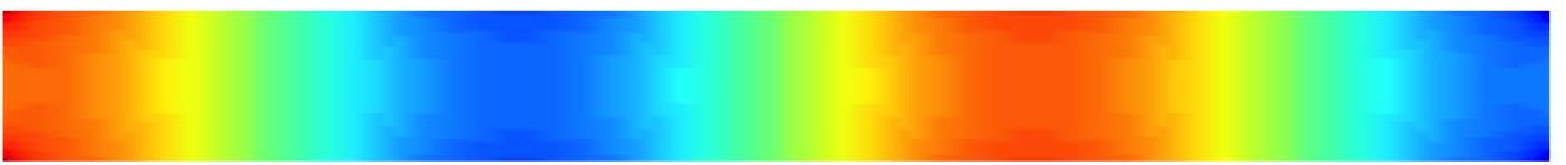} &
\includegraphics[scale=0.45]{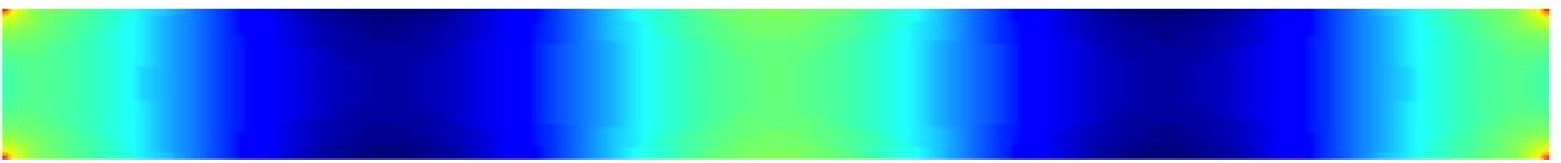} \\
$\sigma_3\simeq 0.066825$ & $\sigma_4\simeq 0.110173$ \\
\includegraphics[scale=0.45]{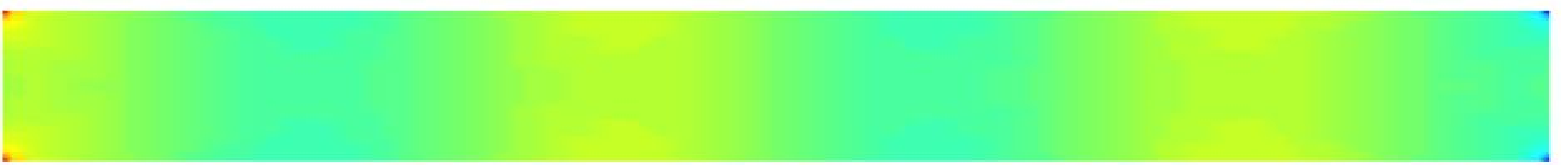} &
\includegraphics[scale=0.45]{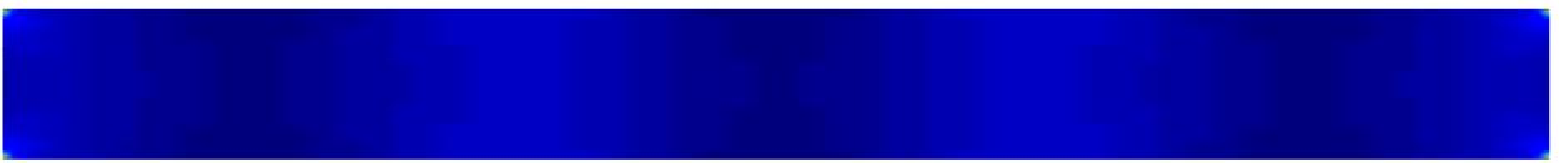} \\
$\sigma_5\simeq 0.156691$ & $\sigma_6\simeq 0.199097$ \\
\end{tabular}
\caption{First six computed Cosserat eigenvectors on the rectangle with aspect ratio $a=0.1$. Same mesh and polynomial degrees as in Fig.~\ref{F5.1} (a).}
\label{F5.2}
\end{figure}

In Figure~\ref{F5.2}, we show the first 6 computed eigenfunctions for an aspect ratio of 10\ee:\ee1. One can see that the first eigenfunctions are almost independent of the transversal variable and look like the corresponding quasimodes $\cos x_1$, $\cos 2x_1$, $\cos 3x_1$,\ldots, although on close inspection, even the first one shows corner singularities. As the eigenvalue grows, the unboundedness at the corners becomes more pronounced, and when the essential spectrum is attained, the corner singularities completely dominate the behavior of the eigenfunction. Thus we see the behavior that was discussed above in Remark \ref{R:sings}.


\subsection{Cuboids}
We show results of computations for domains $\Omega\subset\R^{3}$, along the same lines as in two dimensions. We choose scalar finite dimensional spaces $\gU\subset H^1_0(\Omega)$ and $\gP\subset L^2(\Omega)$. The 3D discrete version of the Schur complement is
\begin{equation}
   \SS(\gU,\gP) = B_1 R^{-1} B_1^\top + B_2 R^{-1} B_2^\top + B_3 R^{-1} B_3^\top 
\end{equation}
where $R$ is still  the stiffness matrix associated with the $\nabla:\nabla$ bilinear form on $\gU\times \gU$ and $B_k$ is the matrix associated with the bilinear form $(u,p)\mapsto\int_\Omega \partial_k u\,p$ on $\gU\times\gP$, now for $k=1,2,3$.

We present computations on elongated cuboids of the form $\frac1a\times$1$\times$1 with $a$ ranging from $0.05$ to $1$, see Figure \ref{F5.3}, compare with \cite{Dobrowolski2005}. From \cite{Dobrowolski2003} we know that for such a ``channel domain'', an upper bound is valid: $\sigma(\Omega)\le \gamma a^2$ where the constant $\gamma$ depends on the cross section of the channel. Here we prove an improvement of this estimate.

\begin{lemma}
\label{L:channel}
Let $\omega$ be a Lipschitz domain in $\R^{d-1}$ ($d\ge2$) and for $a>0$ set $\Omega_a = (0,\frac{\pi}{a})\times\omega$. Denote by $\Delta'$ the Laplacian in $\omega$ and consider the solution $\psi_a$ of the Dirichlet problem
\begin{equation}
\label{E:Dira}
   \psi_a\in H^1_0(\omega),\quad (-\Delta'+a^2)\psi_a = 1\ \ \mbox{in}\ \ \omega.
\end{equation}
Then, with $\mu(\omega)$ the measure of $\omega$, there holds
\begin{equation}
\label{E:channel}
   \sigma(\Omega_a) \le a^2 \frac{\langle\psi_a,1\rangle_\omega}{\mu(\omega)}\,.
\end{equation}
\end{lemma}

\begin{proof}
We denote by $(x_1,x')\in(0,\frac{\pi}{a})\times\omega$ the coordinates in $\Omega_a$ and consider the quasimode $p(x)=\cos ax_1$. Then $\nabla p = (-a\sin ax_1, 0 \ldots 0)^\top$ and we check that
\[
   \Delta^{-1}\nabla p = (a\sin ax_1\,\psi_a(x'), 0 \ldots 0)^\top .
\]
Hence $\SS p = a^2\cos ax_1\,\psi_a(x')$. It is easy to see that the Rayleigh quotient satisfies
\[
   \frac{\langle \SS p, p\rangle}{\langle p, p\rangle} = 
   a^2 \frac{\langle\psi_a,1\rangle_\omega}{\mu(\omega)}\,,
\]
which ends the proof of \eqref{E:channel}.
\end{proof}

\begin{figure}
\begin{tabular}{ccc}
\includegraphics[scale=0.45]{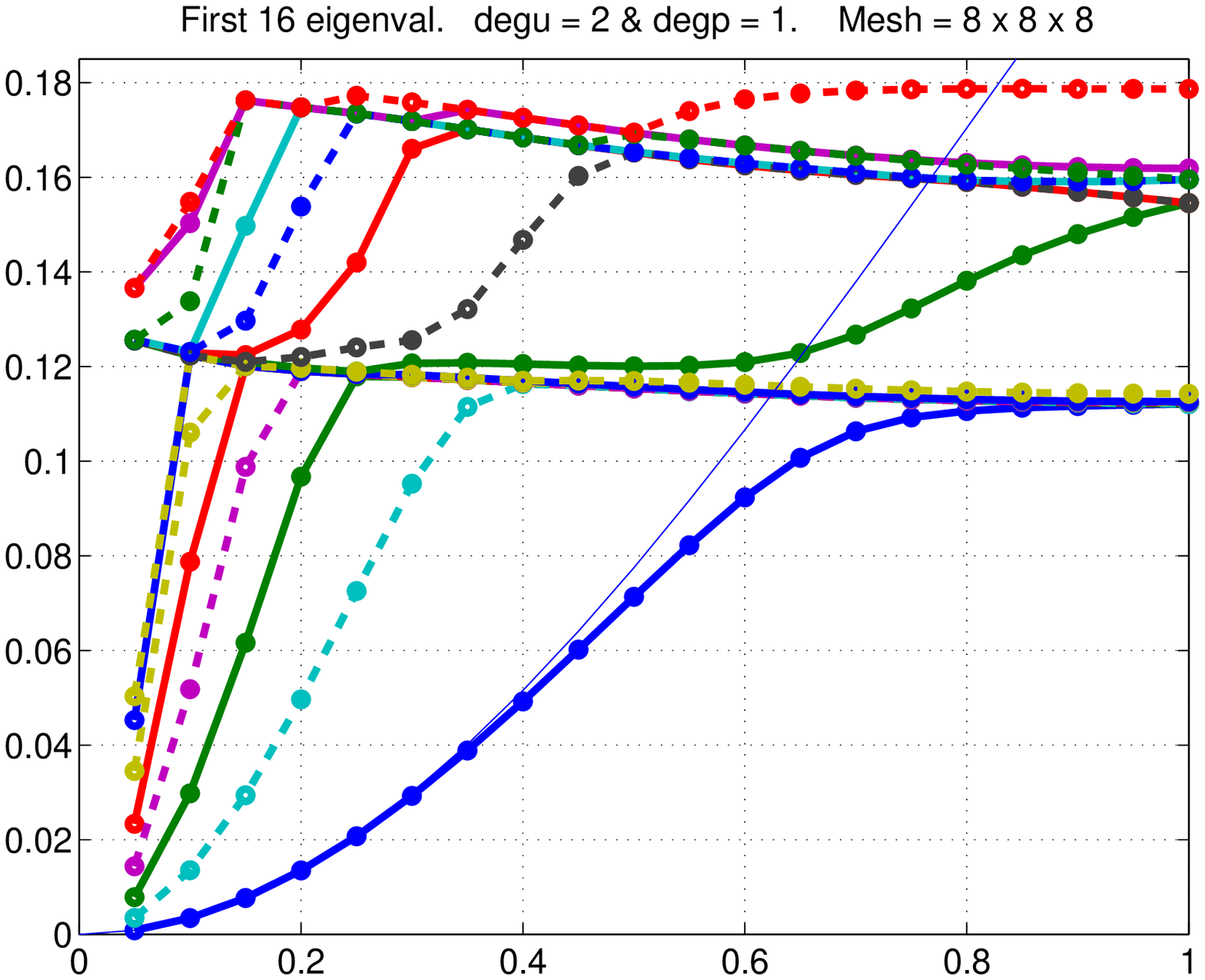} &
\includegraphics[scale=0.45]{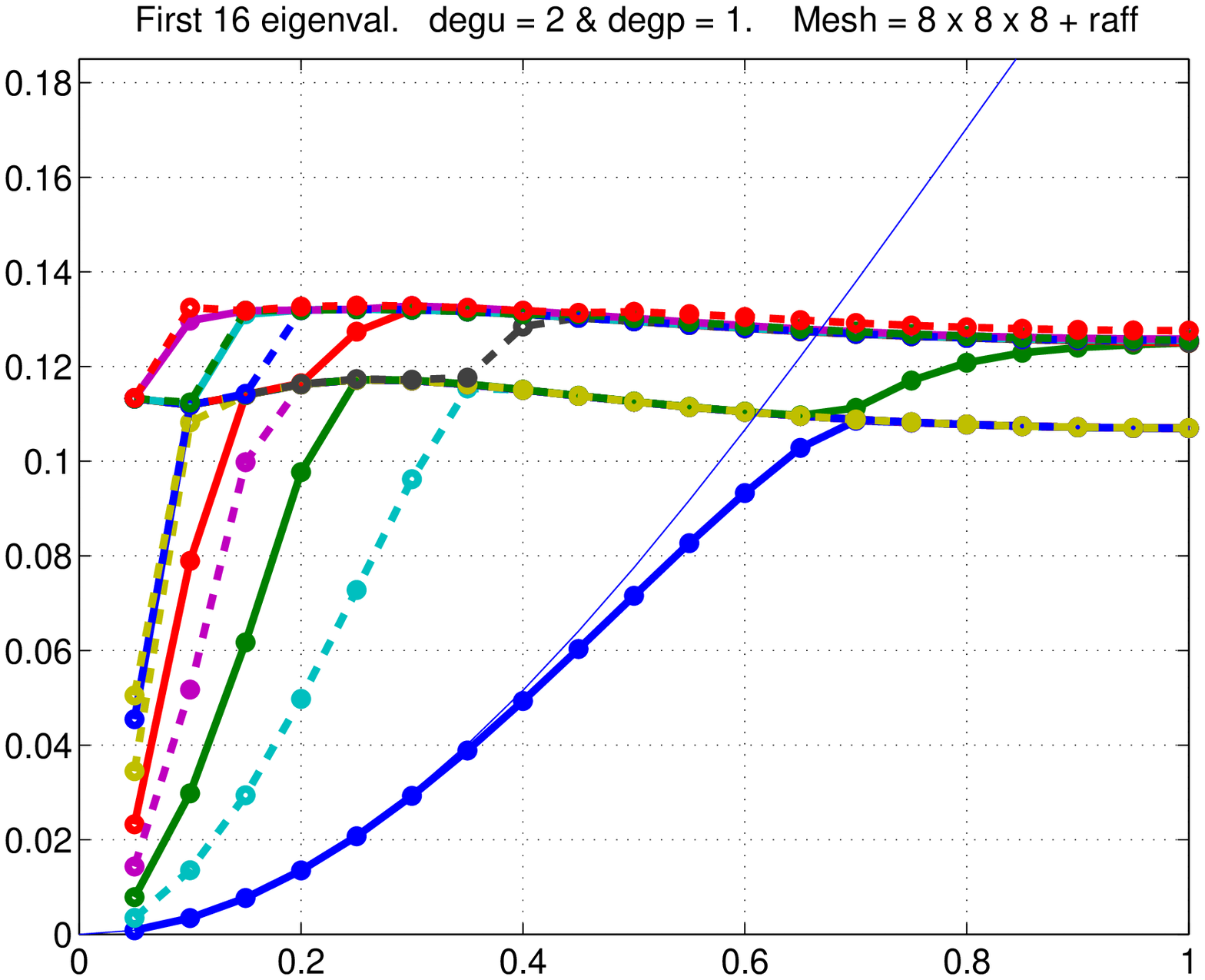} \\
(a) Uniform mesh, $\Q_2$ for $\bu$ and $\Q_1$ for $p$ & 
(b) Tensor refined mesh, $\Q_2$ for $\bu$ and $\Q_1$ for $p$  \\
\end{tabular}
\caption{First 16 computed Cosserat eigenvalues on cuboids $\frac1a\times$1$\times$1 vs parameter $a\in(0,1]$ (abscissa). The solid curve is the upper bound \eqref{E:cuboid}.}
\label{F5.3}
\end{figure}

To obtain an upper bound for the cuboids $\Omega_a$ of dimensions $\frac1a\times$1$\times$1, we take $\omega=(0,\pi)\times(0,\pi)$ in Lemma \ref{L:channel}. We can calculate $\psi_a$ by Fourier expansion in $\omega$: Starting from
\[
   1 = \frac{16}{\pi^2} \sum_{k_1,k_2>0,\ \text{odd}}
   \frac{1}{k_1}\,\frac{1}{k_2} \, \sin k_1x'_1\,\sin k_2x'_2
\]
we find:
\[
   \psi_a(x'_1,x'_2) = \frac{16}{\pi^2} \sum_{k_1,k_2>0,\ \text{odd}}
   \frac{1}{k_1}\,\frac{1}{k_2}\,\frac{1}{k^2_1+k^2_2+a^2} \, \sin k_1x'_1\,\sin k_2x'_2
\]
Hence
\[
   \langle\psi_a,1\rangle_\omega = \frac{64}{\pi^2} \sum_{k_1,k_2>0,\ \text{odd}}
   \frac{1}{k^2_1}\,\frac{1}{k^2_2}\,\frac{1}{k^2_1+k^2_2+a^2} 
\]
and \eqref{E:channel} yields for our cuboids
\begin{equation}
\label{E:cuboid}
   \sigma(\Omega_a) \le \left(\frac{8\,a}{\pi^2} \right)^2 \sum_{k_1,k_2>0,\ \text{odd}}
   \frac{1}{k^2_1}\,\frac{1}{k^2_2}\,\frac{1}{k^2_1+k^2_2+a^2} \ .
\end{equation}

\begin{remark}
\label{R:channel}
Translated into our notation, Dobrowolski's result \cite[\S3]{Dobrowolski2003} provides for $\Omega_a=(0,\frac{\pi}{a})\times\omega$
\begin{equation}
\label{E:channelD}
   \sigma(\Omega_a) \le \left(\frac{2\sqrt{3}\,a}{\pi} \right)^2\,
   \frac{\langle\psi_0,1\rangle_\omega}{\mu(\omega)}\,,
\end{equation}
hence for the cuboid $\Omega_a=(0,\frac{\pi}{a})\times(0,1)\times(0,1)$
\begin{equation}
\label{E:cuboidD}
   \sigma(\Omega_a) \le \left(\frac{16\sqrt{3}\,a}{\pi^3} \right)^2 \sum_{k_1,k_2>0,\ \text{odd}}
   \frac{1}{k^2_1}\,\frac{1}{k^2_2}\,\frac{1}{k^2_1+k^2_2} \ .
\end{equation}
\end{remark}

We observe in Figure \ref{F5.3} that our computations are in good agreement with the bound \eqref{E:cuboid}. 
Nevertheless, they are more difficult to interpret than in 2D. The possible presence of spurious eigenvalues is not easy to distinguish from the manifestation of the essential spectrum. Moreover, at this stage, it is an open question whether the bottom of the essential spectrum comes from edges or from corners. Numerical experiments on cylinders with circular or annular sections, which are compatible with the upper bound \eqref{E:channel}, tend to suggest that the essential spectrum coming from the edges is restricted to $[\frac12-\frac1\pi,\frac12+\frac1\pi]$. Therefore we may conjecture that we see in Figure~\ref{F5.3} a manifestation of the essential spectrum coming from the corners of the cuboids.


\begin{thebibliography}{10}

\bibitem{ChizhOlsh00}
{\sc E.~V. Chizhonkov and M.~A. Olshanskii}, {\em On the domain geometry
  dependence of the {LBB} condition}, M2AN Math. Model. Numer. Anal., 34
  (2000), pp.~935--951.

\bibitem{Cosserat1898b}
{\sc E.~{Cosserat} and F.~Cosserat}, {\em {Sur la d\'eformation infiniment
  petite d'un ellipso\"ide \'elastique.}}, {C. R. Acad. Sci., Paris}, 127
  (1898), pp.~315--318.

\bibitem{Cosserat1898}
{\sc E.~{Cosserat} and F.~{Cosserat}}, {\em {Sur les \'equations de la
  th\'eorie de l'\'elasticit\'e.}}, {C. R. Acad. Sci., Paris}, 126 (1898),
  pp.~1089--1091.

\bibitem{Cosserat1902}
{\sc E.~{Cosserat} and F.~{Cosserat}}, {\em {Sur la d\'eformation infiniment
  petite d'une enveloppe sph\'erique \'elastique.}}, {C. R. Acad. Sci., Paris},
  133 (1902), pp.~326--329.

\bibitem{CostabelDauge93c}
{\sc M.~Costabel and M.~Dauge}, {\em Construction of corner singularities for
  {Agmon-Douglis-Nirenberg} elliptic systems}, Math. Nachr., 162 (1993),
  pp.~209--237.

\bibitem{CoDaLausanne2000}
{\sc M.~Costabel and M.~Dauge}, {\em On the {Cosserat} spectrum in polygons and
  polyhedra}.
\newblock Talk at a Conference in Lausanne, 2000.

\bibitem{CoDa_IBAFHP}
\leavevmode\vrule height 2pt depth -1.6pt width 23pt, {\em On the inequalities
  of {B}abu\v{s}ka--{A}ziz, {F}riedrichs and {Horgan--Payne}}, tech. rep.,
  Institut de Recherche Math{\'e}matique de Rennes,
  \href{http://arxiv.org/abs/1303.6141} {http://arxiv.org/abs/1303.6141}, 2013.

\bibitem{Crouzeix1997}
{\sc M.~Crouzeix}, {\em {On an operator related to the convergence of {U}zawa's
  algorithm for the Stokes equation.}}, in {Computational science for the 21st
  century}, M.-O. Bristeau, G.~Etgen, W.~Fitzgibbon, J.~Lions, J.~P\'eriaux,
  and M.~Wheeler, eds., {Chichester: John Wiley \& Sons}, 1997, pp.~242--249.

\bibitem{Dauge88}
{\sc M.~Dauge}, {\em Elliptic Boundary Value Problems in Corner Domains --
  Smoothness and Asymptotics of Solutions}, Lecture Notes in Mathematics, Vol.
  1341, Springer-Verlag, Berlin, 1988.

\bibitem{Dobrowolski2003}
{\sc M.~Dobrowolski}, {\em On the {LBB} constant on stretched domains}, Math.
  Nachr., 254/255 (2003), pp.~64--67.

\bibitem{Dobrowolski2005}
\leavevmode\vrule height 2pt depth -1.6pt width 23pt, {\em On the {LBB}
  condition in the numerical analysis of the {S}tokes equations}, Appl. Numer.
  Math., 54 (2005), pp.~314--323.

\bibitem{Friedrichs1937}
{\sc K.~Friedrichs}, {\em On certain inequalities and characteristic value
  problems for analytic functions and for functions of two variables}, Trans.
  Amer. Math. Soc., 41 (1937), pp.~321--364.

\bibitem{Horgan1975}
{\sc C.~O. Horgan}, {\em Inequalities of {K}orn and {F}riedrichs in elasticity
  and potential theory}, Z. Angew. Math. Phys., 26 (1975), pp.~155--164.

\bibitem{HorganPayne1983}
{\sc C.~O. Horgan and L.~E. Payne}, {\em On inequalities of {K}orn,
  {F}riedrichs and {B}abu\v ska-{A}ziz}, Arch. Rational Mech. Anal., 82 (1983),
  pp.~165--179.

\bibitem{Kondratev67}
{\sc V.~A. Kondrat'ev}, {\em Boundary-value problems for elliptic equations in
  domains with conical or angular points}, Trans. Moscow Math. Soc., 16 (1967),
  pp.~227--313.

\bibitem{KozlovMazyaRossmann97b}
{\sc V.~A. Kozlov, V.~G. Maz'ya, and J.~Rossmann}, {\em Elliptic boundary value
  problems in domains with point singularities}, Mathematical Surveys and
  Monographs, 52, American Mathematical Society, Providence, RI, 1997.

\bibitem{KozlovMazyaRossmann01}
\leavevmode\vrule height 2pt depth -1.6pt width 23pt, {\em Spectral Problems
  Associated with Corner Singularities of Solutions to Elliptic Equations},
  Mathematical Surveys and Monographs, 85, American Mathematical Society,
  Providence, RI, 2001.

\bibitem{KozlovMazyaSchwab91}
{\sc V.~A. Kozlov, V.~G. Maz'ya, and C.~Schwab}, {\em On singularities of
  solutions to the boundary value problems near the vertex of a rotational
  cone}, Report LiTH-MAT-R-91-24, Link\"oping University, 1991.

\bibitem{KozlovMazyaSchwab92}
\leavevmode\vrule height 2pt depth -1.6pt width 23pt, {\em On singularities of
  solutions of the displacement problem of linear elasticity near the vertex of
  a cone.}, Arch. Rational Mech. Anal., 119 (1992), pp.~197--227.

\bibitem{LiuMarkenscoff1998}
{\sc W.~Liu and X.~Markenscoff}, {\em The discrete {C}osserat eigenfunctions
  for a spherical shell}, J. Elasticity, 52 (1998/99), pp.~239--255.

\bibitem{Malkus1981}
{\sc D.~S. Malkus}, {\em Eigenproblems associated with the discrete {LBB}
  condition for incompressible finite elements}, Internat. J. Engrg. Sci., 19
  (1981), pp.~1299--1310.

\bibitem{MazyaRossmann10}
{\sc V.~Maz'ya and J.~Rossmann}, {\em Elliptic equations in polyhedral
  domains}, vol.~162 of Mathematical Surveys and Monographs, American
  Mathematical Society, Providence, RI, 2010.

\bibitem{MazyaPlamenevskii80b}
{\sc V.~G. Maz'ya and B.~A. Plamenevskii}, {\em {$L^p$} estimates of solutions
  of elliptic boundary value problems in a domain with edges}, Trans. Moscow
  Math. Soc., 1 (1980), pp.~49--97.
\newblock Russian original in \emph{Trudy Moskov. Mat. Obshch.} \textbf{37}
  (1978).

\bibitem{Mikhlin1973}
{\sc S.~G. Mihlin}, {\em The spectrum of the pencil of operators of elasticity
  theory}, Uspehi Mat. Nauk, 28 (1973), pp.~43--82.

\bibitem{Oden_perso}
{\sc J.~T. Oden}, {\em Personal communication}.
\newblock Mafelap conference Uxbridge, 2013.

\bibitem{SimadervWahl2006}
{\sc C.~G. Simader and W.~von Wahl}, {\em Introduction to the {C}osserat
  problem}, Analysis (Munich), 26 (2006), pp.~1--7.

\bibitem{Stoyan1999}
{\sc G.~Stoyan}, {\em Towards discrete {V}elte decompositions and narrow bounds
  for inf-sup constants}, Comput. Math. Appl., 38 (1999), pp.~243--261.

\bibitem{Stoyan2001}
\leavevmode\vrule height 2pt depth -1.6pt width 23pt, {\em Iterative {S}tokes
  solvers in the harmonic {V}elte subspace}, Computing, 67 (2001), pp.~13--33.

\bibitem{Tartar_NS2006}
{\sc L.~Tartar}, {\em An introduction to {N}avier-{S}tokes equation and
  oceanography}, vol.~1 of Lecture Notes of the Unione Matematica Italiana,
  Springer-Verlag, Berlin, 2006.

\bibitem{Zsuppan2004}
{\sc S.~Zsupp{\'a}n}, {\em On the spectrum of the {S}chur complement of the
  {S}tokes operator via conformal mapping}, Methods Appl. Anal., 11 (2004),
  pp.~133--154.

\end{thebibliography}
\def\cprime{$'$} \newcommand{\noopsort}[1]{}\def\cprime{$'$}

\bigskip\bigskip

\end{document}